\documentclass[11pt,reqno]{amsart}
\usepackage{graphicx}
\usepackage{float}
\usepackage[caption = false]{subfig}
\usepackage{framed}
\usepackage{enumerate}
\usepackage{mathtools}
\usepackage{breqn}
\usepackage{framed}
\usepackage{array, geometry, graphicx}
\usepackage{amsmath,amsfonts,paralist,amssymb,amsthm,mathrsfs}
\usepackage[export]{adjustbox}
\usepackage{setspace}
\usepackage{pdfpages}
\usepackage{wasysym}
\usepackage[skip=1ex, labelfont=bf, font=footnotesize]{caption}
\usepackage{lipsum}
\input{insbox}
\newtheorem{example}{Example}
\newtheorem{theorem}{Theorem}[section]
\newtheorem{corollary}[theorem]{Corollary}

\theoremstyle{definition}
\newtheorem{definition}[theorem]{Definition}
\theoremstyle{remark}

\numberwithin{equation}{section}

\DeclareMathOperator{\sech}{sech}
\DeclareMathOperator{\csch}{csch}

\theoremstyle{theorem}

\newtheorem{theo}{Lemma}
\newcounter{tmp}

\begin{document}
	
\title[Briot-Bouquet Differential Subordination]{On Sufficient conditions  for the class $\mathcal{S}^{*}_{\cosh \sqrt{z}}$}

	\thanks{The second author is supported by Delhi Technological University, New Delhi}
 	\author{S. Sivaprasad Kumar}
	\address{Department of Applied Mathematics, Delhi Technological University, Delhi--110042, India}
\email{spkumar@dce.ac.in}
	\author{Mridula Mundalia}
	\address{Department of Applied Mathematics, Delhi Technological University, Delhi--110042, India}
\email{mridulamundalia@yahoo.co.in}

	\subjclass[2010]{30C45, 30C80}
	
	\keywords{Univalent functions; Starlike functions; Differential subordination; Briot-Bouquet differential subordination; Hyperbolic cosine function}
\begin{abstract}
Using differential subordination technique, such as Briot-Bouquet and others, we establish sufficient conditions for functions to be in a class $\mathcal{S}^{*}_{\varrho},$ consisting of  starlike functions that are associated with $\varrho(z):=\cosh \sqrt{z}.$ 
Furthermore, by employing admissibility conditions, we obtain various differential subordination results pertaining to the class $\mathcal{S}^{*}_{\varrho}.$


       
\end{abstract}

\maketitle
	
\section{Introduction}

  Let $\mathcal{A}$ be the class of all analytic functions, defined on the open unit disc $\mathbb{D}:=\left\{z\in\mathbb{C}:|z|<1\right\},$ of the form $f(z)=z+a_{2}z^{2}+a_{3}z^{3}+\cdots .$ Let $\mathcal{S}$ denote the class of all univalent functions in $\mathcal{A}.$ Given two analytic functions $f(z)$ and $g(z)$ in $\mathbb{D},$  $f(z)$ is said to be subordinate to $g(z),$ symbolically $f\prec g,$ if there exists a Schwarz function $w(z)$ such that $w(0)=0$ and $f(z)=g(w(z)).$ Especially, if $g(z)$ is a univalent function in $\mathbb{D}$, then $f\prec g$ if and only if the two conditions 
$f(0)=g(0) \text{ and } f(\mathbb{D})\subset g(\mathbb{D})$ are met.
In 1992, Ma and Minda \cite{Ma & Minda} investigated a special subclass of $\mathcal{S},$ namely $\mathcal{S}^{*}(\phi),$ given by
\[\mathcal{S}^{*}(\phi)=\left\{f\in\mathcal{A}:\frac{zf'(z)}{f(z)}\prec \phi(z), \;z\in\mathbb{D}\right\},\]
where certain conditions are imposed on the analytic function $\phi(z),$ such as:
$\phi(z)$ is starlike with respect to $\phi(0)=1,$  $\phi(z)$ is univalent in $\mathbb{D}$ such that $\phi'(0)>0,$ $\phi(z)=\overline{\phi(\bar{z})}$ and $\operatorname{Re} \phi(z)>0,$ for each $z\in\mathbb{D}.$ 
The class $\mathcal{S}^{*}(\phi)$ has been extensively studied by various authors for different choices of $\phi(z),$ see \cite{Cho & Virender(2019), Ebadian(2020),Goel & Sivaprasad(2020),kumar & Ravichandran(2013)} and the references therein. Moreover, these subclasses represent instances, where the properties of $\phi(z)$ have been tailored to investigate certain geometrical properties of functions lying in the class $\mathcal{S}^{*}(\phi)$. We provide a table (see Table \ref{coshsubdnTable1}), enlisting some subclasses of starlike functions, obtained for special choices of $\phi(z).$ 
The classes $\mathcal{S}^{*}_{L},$ $\mathcal{S}^{*}[A,B],$ $\mathcal{S}^{*}_{s},$ $\mathcal{S}^{*}_{e}$ and $\mathcal{S}^{*}_{\rightmoon}$ (see Table \ref{coshsubdnTable1}) have garnered significant attention in the past, notably in the works by \cite{Aouf n Sokol, Janowski, Masih n Kanas, Mendiratta n Nagpal,Raina}. 
Inspired by the Ma and Minda subclasses of starlike functions, recently we introduced and studied the following subclass of starlike functions associated with a hyperbolic cosine function \cite{Mundalia(2022)}, defined as
\begin{equation*}
\mathcal{S}^{*}_{\varrho}:=\left\{f\in\mathcal{A}:\frac{zf'(z)}{f(z)}\prec \cosh \sqrt{z}=:\varrho(z) ,\text{ }z\in\mathbb{D}\right\},
\end{equation*}
where the branch of the square root function is so chosen such that \begin{equation}\label{e57}
\cosh\sqrt{z}=1+\frac{z}{2!}+\frac{z^{2}}{4!}+\cdots=\sum_{n=0}^{\infty}\frac{z^{n}}{(2n)!}.
\end{equation} 
Note that the conformal mapping $\varrho:\mathbb{D}\to \mathbb{C},$ maps the unit disc $\mathbb{D}$ onto the region
\begin{align}\label{e65}
\Omega_{\varrho}:=\{\omega\in\mathbb{C}:|\log(\omega+\sqrt{\omega^{2}-1})|^{2}<1\},
\end{align}
defined on the principle branch of logarithmic and square root function.
Miller and Mocanu extensively studied the theory of differential subordination, see \cite{Miller & Mocanu}, apart from others \cite{Bohra n Ravi (2021),Ebadian(2020),Goel & Sivaprasad(2020),Keong(2013),Miller & Mocanu,Sharma n Ravi(2021),Sukhjit Singh(2009)}. Recently, several authors have established conditions on $\eta,$ so that $1+\eta z p'(z)/p^{n}(z)\prec \phi(z)$ $(n=0,1,2)$ implies $p(z)\prec \psi(z),$ these conditions vary depending on the choice of $\phi(z)$ and $\psi(z)$ such as  $2/(1+e^{-z}),$ $\sqrt{1+z},$ $(1+Az)/(1+Bz)$ $(-1\leq B<A\leq 1)$ and $e^{z},$ see \cite{Goel & Sivaprasad(2020),Janowski,kumar & Ravichandran(2013),Mendiratta n Nagpal}. 
Further, the regions corresponding to functions $e^{z},$ $z+\sqrt{1+z^{2}},$ $(1+Az)/(1+Bz)$  and $(1+sz)^{2},$  respectively are 
\begin{align}
 \Omega_{e}:&=\{\omega\in\mathbb{C}:|\log \omega|<1\},\nonumber \\
		\Omega_{\rightmoon}:&=\{\omega\in\mathbb{C}:|\omega^{2}-1|<2|w|\}\nonumber\\
  &=\{\omega\in\mathbb{C}:|\omega-1|< \sqrt{2}\} \cap \{\omega\in\mathbb{C}: |\omega+1|> \sqrt{2}\}=\Delta_{1}\cap \Delta_{2},\label{crescentregion}\\
  \Omega_{A,B}:&=\{w\in\mathbb{C}:|\omega-1|<|A-B \omega|\}\nonumber \\
  \Omega_{s}:&=\{u+i v:((u-1)^{2}+v^{2}-s^{4})^{2}<4 s^{2}((u-1+s^{2})^{2}+v^{2})\}\nonumber\\
  &\quad \subset \{\omega\in\mathbb{C}:|\omega-1|< |s|(|s|+2)\} \label{e66}.
\end{align}


\begin{table}[ht]
	\renewcommand{\arraystretch}{1.18}
	\caption{List of some notable classes of starlike functions} 
	\centering 
	\begin{tabular}{|l|l|l|l|} 
		\hline 
		{\bf{$\mathcal{S}^{*}(\phi)$}}  & \textbf{$\phi(z)$}  & \textbf{References}     \\ [0.9ex] 
		\hline  
		$\mathcal{S}^{*}_{e}$       &   $\phi_{e}(z):=e^{z}$ &   \cite{Mendiratta n Nagpal} Mendiratta et al.  \\ 
		$\mathcal{S}^{*}_{L}$    &   $\phi_{L}(z):=\sqrt{1+z}$  &  \cite{Sokol n Stankiewicz} Sok\'{o}\l \ et al.  \\ 
		$\mathcal{S}^{*}_{\rightmoon}$     &  $\phi_{\rightmoon}(z):=z+\sqrt{1+z^2}$   &     \cite{Raina} Raina et al. \\  
		$\mathcal{S}^{*}_{s}$       &   $\phi_{s}(z):=(1+sz)^{2},$ $-1/\sqrt{2} \leq s \leq 1/\sqrt{2},$ $s\neq 0$  &    \cite{Masih n Kanas} Masih et al.  \\ 
		$\mathcal{S}^{*}[A,B]$   &  $\phi_{A,B}(z):=(1+Az)/(1+Bz),$ $-1\leq B<A\leq 1$  & \cite{Janowski}  W. Janowski \\ 
		\hline 
	\end{tabular}
	\label{coshsubdnTable1}
\end{table}
 \InsertBoxR{-1}{\begin{minipage}{0.4\linewidth} 
      \begin{center}
      \hspace*{0cm}
        \includegraphics[width=\linewidth]{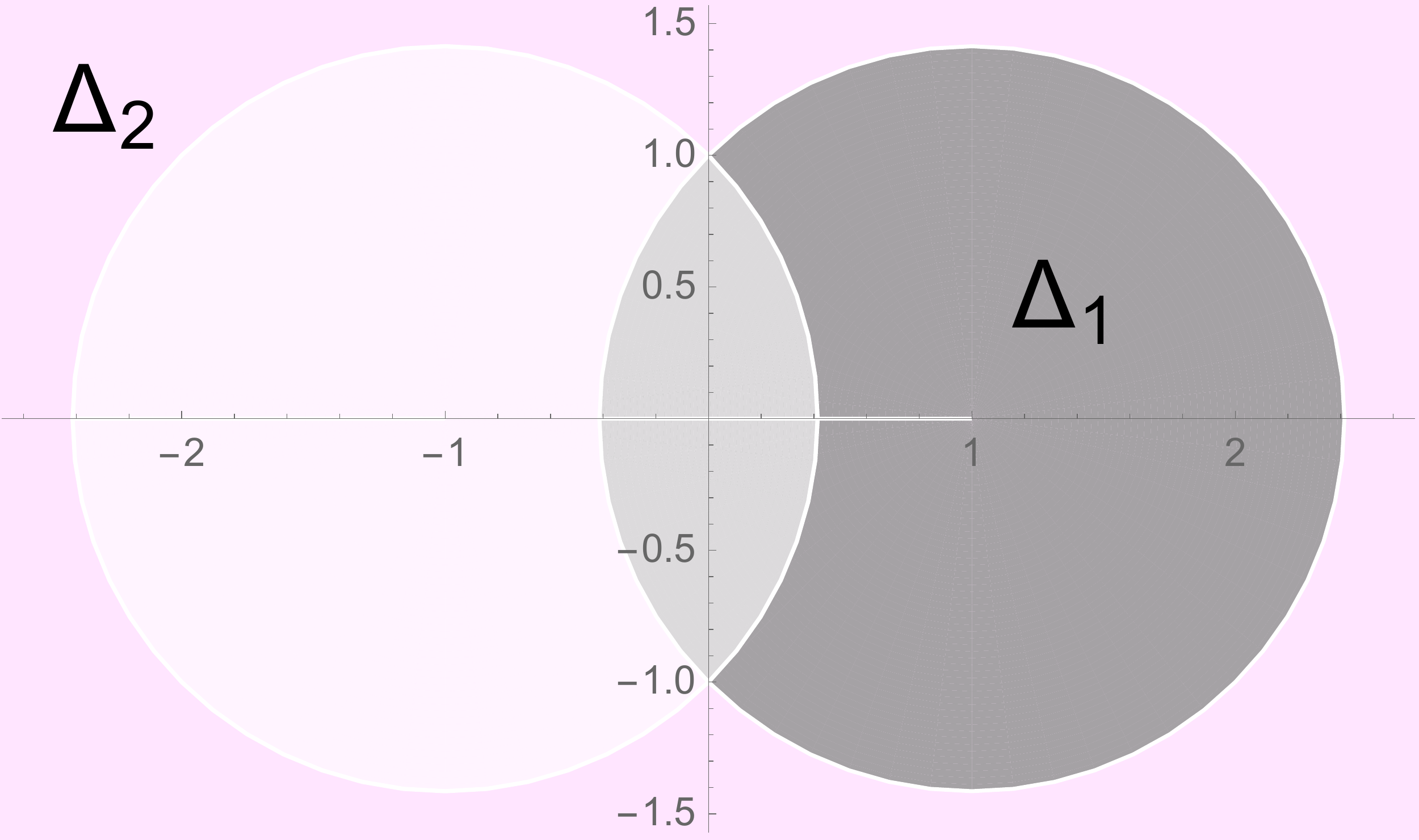}
      \end{center}
      \captionof{figure}{$\Omega_{\rightmoon}=\Delta_{1}\cap\Delta_{2}.$} 
\label{fig9}
                   \end{minipage}%
                   }[5]  
From Figure~\ref{fig9}, we see that the above crescent domain is given by $\Omega_{\rightmoon}=\Delta_{1}\cap\Delta_{2},$ where $\Delta_{1}=\{w:|w-1|<\sqrt{2}\}$ and $\Delta_{2}=\{w:|w+1|>\sqrt{2}\}.$ Generally,  while dealing with the crescent domain, we conclude that if $x\notin\Delta_{i},$ then $x\notin\Omega_{\rightmoon},$ as $\Omega_{\rightmoon}\subset \Delta_{i},$ where $i=1,2.$

In the present investigation, we explore sufficient conditions for functions to be in $\mathcal{S}^{*}_{\varrho}$ by establishing Briot-Bouquet differential subordination implications with dominants such as $e^{z},$ $(1+Az)/(1+Bz),$ $z+\sqrt{1+z^{2}}$ and $(1+sz)^{2}.$ Additionally, we derive some first-order differential subordination results for $\mathcal{S}^{*}_{\varrho}$ and diagrammatically validate the sharpness of our findings. Finally, we deduce certain admissibility results for  $\mathcal{S}^{*}_\varrho,$ accompanied by some applications  and illustrations of our findings.



\section{Briot-Bouquet Differential Subordination Results}
\noindent Numerous studies have delved into the Briot-Bouquet differential subordination, given by 
\begin{equation}\label{e63}
	p(z)+\frac{zp'(z)}{\eta p(z) + \gamma}\prec h(z),
\end{equation} 
with contributions from various authors over time. For comprehensive insights, refer to \cite{Bohra n Ravi (2021),Keong(2013),Miller & Mocanu}.
Specifically, one can refer to the works of Ravichandran et al. \cite{Sharma n Ravi(2021)} and Singh et al. \cite{Sukhjit Singh(2009)} for notable contributions to Briot-Bouquet differential subordination results.
This unique form of differential subordination holds significant importance in univalent function theory and has a wide range of applications.
Additionally,  it's common to analyze the implication results related to \eqref{e63} with the assumption that $h(z)$ is a convex function and $\operatorname{Re}(\eta h(z)+\gamma)>0.$ Following a similar approach, we derive Briot-Bouquet differential subordination results for the class $\mathcal{S}^{*}_{\varrho}$.  
To proceed, we require the following lemma, which will be instrumental in establishing some of our main results in this section. 
\begingroup
\setcounter{tmp}{\value{theo}}
\setcounter{theo}{0} 
\renewcommand\thetheo{\Alph{theo}}
\begin{theo}\cite[Lemma 1.3, p.28] {Ruscheweyh} \label{l8}
Let $w$ be a meromorphic function in $\mathbb{D},$ $w(0)=0.$ If for some $z_{0}\in\mathbb{D},$ $\displaystyle{\max _{|z|\leq |z_{0}|} }|w(z)|=|w(z_{0})|,$ then it follows that $z_{0}w'(z_{0})/w(z_{0}) \geq 1.$
\end{theo}
\endgroup
Now we begin with the following theorem: 
\begin{theorem}\label{t6}
Let $\eta,\gamma\in\mathbb{R}$ such that $\gamma\neq -\eta,$ 
satisfy any of the following conditions: 
\begin{enumerate}[(i)]
 \item   For $\phi(z)=z+\sqrt{1+z^{2}},$ we have 
 \begin{align}\label{e29}
		 \eta_{1} \leq   \eta     \leq \eta_{0} \text{ and }  
	\eta_{2} \leq   \eta  \leq \eta_{3}, 
	\end{align} 
 where \begin{gather*}
		\eta_{0}=-\frac{\gamma}{\cosh 1}+\frac{\sinh 1}{(2 \cosh 1 (1+\sqrt{2}-\cosh 1))}, \text{ }  \eta_{1}=-\frac{\gamma}{\cosh 1} \nonumber, \\
			\eta_{2}=-\frac{\gamma}{\cos 1} -\frac{\sin 1}{(2 \cos 1 (1+\sqrt{2}-\cos 1))},  \text{  }
			\eta_{3}=-\frac{\gamma}{  \cos 1},
			\end{gather*}
 \item For $\phi(z)=(1+sz)^{2}$ and $0<s\leq 1/\sqrt{2},$ we have  
\begin{align} \label{e43}
		 \eta_{3}+ \eta_{2} \leq \eta \leq \eta_{2} \text{ and }  \left.\begin{array}{lll} 
			\eta \geq \eta_{0}   &  \text{if }  0 < s \leq -1+\sqrt{\cosh 1}, \\ 
			\eta_{0} \leq \eta\leq   \eta_{0}+\eta_{1}  & \text{if }  -1+\sqrt{\cosh 1}< s \leq 1/\sqrt{2},  \end{array} \right\}
	\end{align}
	where 
	\begin{align*}
		& \eta_{0}=-\dfrac{\gamma}{\cosh 1}, \text{ } \eta_{1}= \dfrac{\sinh 1}{2\cosh 1((1+s)^{2}-\cosh 1)},\\&
		\eta_{2}= -\frac{\gamma}{\cos 1}, \text{ } \eta_{3} = - \dfrac{\sin 1}{2\cos 1((1+s)^{2}-\cos 1)}.
		\end{align*}  
 \item For $\phi(z)=e^{z},$ we have  \begin{align}\label{e28}
	\eta_{1} <  \eta   \leq  \eta_{0} \text{ and } \eta_{2} \leq  \eta   < \eta_{3},
	\end{align}  
  where \begin{gather*}
		\eta_{0}=-\frac{\gamma}{\cosh 1}+ \frac{\sinh 1}{2\cosh 1(e - \cosh 1)}, \text{ } 
			\eta_{2}= -\frac{\gamma}{\cos 1}-\frac{\sin 1}{2\cos 1(e -\cos 1)},
		\end{gather*}
$\eta_{1}$ and $\eta_{3}$ are as given in (i). 
\end{enumerate}
If $p(z)$ is an analytic function, such that  $p(0)=1$ and satisfies 
\begin{equation}\label{e59}
p(z)+\frac{zp'(z)}{\eta p(z) + \gamma} \prec \phi(z),
\end{equation}
	then  
	$p(z)\prec \cosh \sqrt{z}.$
\end{theorem}
\begin{proof}
	Let $\mathfrak{B}(z)$ and $w(z)$ be as given below:
\begin{equation}\label{e64}
	\mathfrak{B}(z):=p(z)+\frac{z p'(z)}{\eta p(z) + \gamma} \text{ and } w(z)=(\cosh^{-1}{p(z)})^{2}, 
\end{equation}then we have $p(z)=\cosh\sqrt{w(z)}.$  It is evident that $w(z)$ is a well-defined analytic function, with $w(0)=0.$ 
Now to prove $p(z)\prec \cosh \sqrt{z},$ we need to show that $|w(z)|<1$ in $\mathbb{D}.$  For if,  there exists $z_{0}\in\mathbb{D}$ such that $\max _{|z|\leq |z_{0}|} |w(z)|=|w(z_{0})|=1,$  then by Lemma \ref{l8}, we have $z_{0}w'(z_{0})=k w(z_{0}),$ where $k\geq 1.$ Let $w(z_{0})=e^{2it },$ where $-\pi/2 \leq t\leq \pi/2.$ \\ \ \\
(i)  Let $\phi(z)=z+\sqrt{1+z^{2}},$ then from \eqref{crescentregion}, we have $\phi(\mathbb{D})\subset\{w:|w-1|<\sqrt{2}\}.$ Now we deduce a contradiction by arriving at  $|\mathfrak{B}(z_{0})-1|^{2}\geq 2.$ To do this, we expand the following expression using \eqref{e64}: 
\begin{align}\label{e40}
			|\mathfrak{B}(z_{0})-1|^{2}& = \left|\cosh  \sqrt{w(z_{0})} -1 + \frac{z_{0}w'(z_{0})\sinh \sqrt{w(z_{0})}}{2\sqrt{w(z_{0})}(\eta \cosh \sqrt{w(z_{0})}+\gamma)}\right|^{2}\nonumber \\
			& =\left|\cosh e^{it} - 1 + \frac{k e^{it} \sinh e^{it}}{2 (\eta \cosh e^{it} + \gamma)}\right|^{2}
			\\& =:\frac{N(t)}{D(t)},  
\end{align}
		where  
		\begin{align}\label{e34}
			N(t)&=(\sinh (\cos t)(2 \gamma  k \sin t \cos (\sin t)+2(2 \gamma ^2- \eta ^2) \sin (\sin t)+\eta ^2 \sin (3 \sin t))\nonumber\\&\quad+2 \gamma  \sin (\sin t) \cosh (\cos t) (k \cos t+4 \eta  \cos (\sin t) \sinh (\cos t))+\eta  k (\sin (2 \sin t) \cos t\nonumber\\&\quad+\sin t \sinh (2 \cos t))+\eta ^2 \sin (\sin t) \sinh ^3(\cos t)+3 \eta ^2 \sin (\sin t) \sinh (\cos t) \cosh ^2(\cos t))^2 \nonumber\\&\quad+ 4(\cosh (\cos t) (\eta  k \cos t \sinh (\cos t)-\gamma  k \sin t \sin (\sin t)+2 \gamma  (\gamma -2 \eta ) \cos (\sin t)\nonumber\\&\quad+2 \eta ^2 \sin ^2(\sin t) \cos (\sin t) \sinh ^2(\cos t))-\eta  \cos (\sin t) \cosh ^2(\cos t) (k \sin t \sin (\sin t)\nonumber\\&\quad+2(\eta -2 \gamma ) \cos (\sin t))+\gamma  k \cos t \cos (\sin t) \sinh (\cos t)+2 \eta ^2 \cos ^3(\sin t) \cosh ^3(\cos t)\nonumber\\&\quad +\eta  \sin (\sin t) \sinh ^2(\cos t) (k \sin t \cos (\sin t)-2 \eta  \sin (\sin t))-2 \gamma ^2)^2
		\end{align}	
		and 
		\begin{align}\label{e35}
			D(t)&=	16((\gamma+\eta \cos (\sin t) \cosh (\cos t))^{2}+\eta ^2 \sin ^2(\sin t) \sinh ^2(\cos t))^2.
		\end{align}
		
\noindent Define a function $F(t)$ on the interval $[-\pi/2,\pi/2],$ as 
\begin{equation*}\label{e31}
  F(t)=N(t)-2 D(t).
  \end{equation*} 
  Since $F(t)$ is an even function, it is sufficient to show that $F(t)$ is non-negative in $[0,\pi/2]$ or minimum of $F(t)$ is non-negative in $[0,\pi/2].$ A computation reveals that the minimum of $F(t)$ is obtained either at $t=0$ or $t=\pi/2.$  Now we see that 


\begin{align*} 
			F(0)&= (k(\eta  \sinh 2 + 2 \gamma   \sinh 1)  - 4(1-\cosh 1)\left(\gamma + \eta \cosh 1\right)^2)^2-32 \left(\gamma + \eta \cosh 1\right)^4	
			\intertext{and}
			F\left(\frac{\pi}{2}\right)&= (k(\eta \sin 2 + 2 \gamma \sin 1) - 4(\cos 1 - 1 )(\gamma +\eta \cos 1)^{2})^{2} - 32(\gamma +\eta \cos 1)^{4}. 
		\end{align*}
Since $k\geq 1,$ we have $F(0)\geq
( 4 (\cosh 1 - 1) ( \gamma + \eta  \cosh 1 )^2 + (\eta  \sinh 2+2 \gamma  \sinh 1)^2)^2-32 (\gamma+\eta  \cosh 1)^4=:X(\eta).$ But $X(\eta) \geq 0,$ whenever
$\eta_{1} \leq \eta \leq \eta_{0},$  which implies  $F(0)\geq 0.$ 
Infact, for $k\geq 1,$  we have $F(\pi/2) \geq ((\eta \sin 2 + 2 \gamma \sin 1) - 4(\cos 1 - 1 )(\gamma +\eta \cos 1)^{2})^{2} - 32(\gamma +\eta \cos 1)^{4}=:Y(\eta).$ Since,  
$(\eta \sin 2 + 2 \gamma \sin 1) - 4(\cos 1 - 1 )(\gamma +\eta \cos 1)^{2} \geq 4 \sqrt{2} \left(\gamma + \eta \cos 1\right)^{2},$ whenever $ \eta_{2} \leq \eta \leq \eta_{3},$  therefore, $Y(\eta)\geq 0,$ which implies $F(\pi/2)\geq 0.$ Thus  $|\mathfrak{B}(z_{0})-1|^{2}\geq 2,$ which contradicts \eqref{e59}, hence the result follows at once.
\\ \ \\
(ii) Let $\phi(z)=(1+sz)^{2},$ then  from \eqref{e66}, we have $\phi(\mathbb{D})\subset \{w:|w-1|<s(s+2)\}.$ Now we shall show that $|\mathfrak{B}(z_{0})-1
|\geq s(s+2),$ which leads to the desired contradiction. To achieve this, we use the expansion of  $|\mathfrak{B}(z_{0})-1|^{2},$ as given in \eqref{e40}, with $N(t)$ and $D(t)$ given by \eqref{e34} and \eqref{e35}, respectively. 
Now for each $0<s\leq 1/\sqrt{2},$ define
\begin{equation*}\label{e36}
F_{s}(t)=N(t)-s^{2}(s+2)^{2} D(t), \text{ where } - \pi/2\leq t \leq \pi/2.
\end{equation*}
We observe that $F(t)$ is an even function, consequently, it is suffices to prove that $F(t)\geq 0$ for  $t\in[0,\pi/2].$  
Furthermore, it is  observed that $F(t)$ attains its minimum at either $t=0$ or $\pi/2.$ 
Now for $k\geq 1,$ we have
	\begin{align*} 
		F_{s}(0)&=(4 (\cosh 1 -1) (\gamma+\eta  \cosh 1)^2+ k (\eta  \sinh 2 + 2 \gamma  \sinh 1 ))^2 \nonumber \\& \quad -16 s^{2}(s+2)^{2} (\gamma + \eta  \cosh 1  )^4 
		\intertext{and}
		F_{s}\left(\frac{\pi}{2}\right) & = (4 (\cos 1-1) (\gamma+\eta  \cos 1)^2  -  k(\eta  \sin 2 + 2 \gamma  \sin 1 ))^2 \nonumber \\&  \quad -16 s^{2}(s+2)^{2} ( \gamma + \eta  \cos 1)^4.  \nonumber
	\end{align*}



\noindent For each $k\geq 1,$ it can be easily verified that $F_{s}(0)$ is a monotonically increasing function of $k,$ provided $\eta \geq \eta_{0},$ this implies $F_{s}(0)\geq (4 (\cosh 1 -1) (\gamma+\eta  \cosh 1)^2+ (\eta  \sinh 2 + 2 \gamma  \sinh 1 ))^2 -16 s^{2}(s+2)^{2} (\gamma + \eta  \cosh 1  )^4=:X_{s}(\eta).$  Now $X_{s}(\eta)$ can be written as $X_{s}(\eta)=(X_{{s}_{1}}(\eta)-X_{{s}_2}(\eta))(X_{{s}_{1}}(\eta)+X_{{s}_2}(\eta)),$ where $X_{{s}_1}(\eta):=4 (\cosh 1 -1) (\gamma+\eta  \cosh 1)^2+ (\eta  \sinh 2 + 2 \gamma  \sinh 1 ) -4 s(s+2)(\gamma + \eta  \cosh 1  )^{2}$ and $X_{{s}_2}(\eta):=4 (\cosh 1 -1) (\gamma+\eta  \cosh 1)^2+ (\eta  \sinh 2 + 2 \gamma  \sinh 1 ) +4 s(s+2)(\gamma + \eta  \cosh 1  )^{2}.$ Now, for each $0<s\leq 1/\sqrt{2},$ we need to show that $X_{s}(\eta)\geq 0.$ Observe that, for each $0<s\leq \sqrt{\cosh 1}-1,$ $X_{{s}}(\eta)\geq 0$ if and only if $X_{{s}_1}(\eta)\geq X_{{s}_2}(\eta),$ i.e $\eta  \sinh 2 + 2 \gamma  \sinh 1 + 4 (\cosh 1 - (s+1)^2) (\gamma + \eta  \cosh 1 )^2\geq 0,$ which happens,  whenever $\eta \geq \eta_{0},$ thus   $X_{s}(\eta)\geq 0.$   Infact,  for each $\sqrt{\cosh 1}-1<s\leq 1/\sqrt{2},$ $X_{s}(\eta)\geq 0$ if and only if  the inequality $4 ((s+1)^2 - \cosh 1) (\eta  \cosh 1 + \gamma )^2 \leq  \eta  \sinh 2 + 2 \gamma  \sinh 1$ holds, which is possible whenever $\eta_{0}\leq \eta \leq \eta_{0}+\eta_{1},$ which means $X_{s}(\eta)\geq 0.$  Eventually, for each $k\geq 1$ and $0<s\leq 1/\sqrt{2},$ we have $F_{s}(0)\geq 0.$ Next, for  $\eta\leq \eta_{2},$ it can be seen that  $F_{s}(\pi/2)$ is a monotonically increasing function of $k$ and $F_{s}(\pi/2) \geq (4 (\cos 1-1) (\gamma+\eta  \cos 1)^2  -  (\eta  \sin 2 + 2 \gamma  \sin 1 ))^2 -16 s^{2}(s+2)^{2} ( \gamma + \eta  \cos 1)^4=:Y_{s}(\eta).$ In addition, for each $0<s\leq 1/\sqrt{2},$ we have $Y_{s}(\eta)\geq 0,$ whenever $\eta_{2}+\eta_{3} \leq \eta \leq \eta_{2}.$ Therefore, for each $k\geq 1$ and $0<s\leq 1/\sqrt{2},$  we deduce that $F_{s}(\pi/2)\geq 0.$ Thus $
|\mathfrak{B}(z_{0})-1|^{2}\geq s^{2}(s+2)^{2}.$ Hence, we get a contradiction to the hypothesis, given in \eqref{e59}, which completes the proof. \\ \ \\
(iii) Choose $\phi(z)=e^{z}.$ We prove this result by the method of contradiction, similar to (i) and (ii). To proceed, it is suffices to show that  

	\begin{equation}\label{e23}
		|\log \mathfrak{B}(z_{0})|^{2} \geq 1,
	\end{equation} 
where $\log$ denotes the principle branch of logarithmic function.
Consider, 
\begin{align}\label{e61}
 \mathfrak{B}(z_{0})=\cosh e^{it} + \frac{k e^{it} \sinh e^{it}}{2 (\eta \cosh e^{it} + \gamma)}
 =:U(t)+i V(t),
 \end{align} 
where
 
	\begin{align*}
		U(t)&:=\kappa_{t}^{-1}(\cosh (\cos t) (\gamma  (2 \gamma  \cos (\sin t)-k \sin t \sin (\sin t))+\eta  k \cos t \sinh (\cos t)\\&\quad+2 \eta ^2 \sin ^2(\sin t) \cos (\sin t) \sinh ^2(\cos t))+k \cos (\sin t) \sinh (\cos t) (\gamma  \cos t\\&\quad+\eta  \sin t \sin (\sin t) \sinh (\cos t))+\eta  \cos (\sin t) \cosh ^2(\cos t) (4 \gamma  \cos (\sin t)\\&\quad+2 \eta ^2 \cos ^3(\sin t) \cosh ^3(\cos t)-k \sin t \sin (\sin t)))
		\intertext{and}
		V(t)&:= \kappa_{t}^{-1} (\sinh (\cos t) (2 \gamma  k \sin t \cos (\sin t)+2(2 \gamma ^2- \eta ^2) \sin (\sin t)+\eta ^2 \sin (3 \sin t))\\&\quad+\eta  k (\sin (2 \sin t) \cos t+\sin t \sinh (2 \cos t))+\eta ^2 \sin (\sin t) \sinh ^3(\cos t)\\&\quad+3 \eta ^2 \sin (\sin t) \sinh (\cos t) \cosh ^2(\cos t))+2 \gamma  \sin (\sin t) \cosh (\cos t) (k \cos t\\&\quad+4 \eta  \cos (\sin t) \sinh (\cos t)),
	\end{align*}
	with $\kappa_{t}:=2 ((\gamma +\eta  \cos (\sin t) \cosh (\cos t))^2+\eta ^2 \sin ^2(\sin t) \sinh ^2(\cos t)).$ \\ Assume 
 $F(t)=4|\log  \mathfrak{B}(z_{0})|^{2} - 4.$ Now, to prove \eqref{e23}, in view of \eqref{e61}, it is enough to show that 
\begin{align*}\label{e32}
  F(t)=\log^{2}(U^{2}(t)+V^{2}(t))+4 \left(\tan^{-1}\frac{V(t)}{U(t)}\right)^{2}-4\geq 0.
\end{align*}


\noindent		Since $F(-t)=F(t),$ for each $t\in[-\pi/2,\pi/2],$ so we confine our findings to the interval $[0,\pi/2].$ It can be easily verified that $F(t)$ attains its minimum at $t=0$ or $\pi/2.$
\noindent Now for $k\geq 1,$ we have
	\begin{align*}
		F(0)&= \left(\log \left(\cosh 1 + \frac{  k (\eta  \sinh 2 + 2 \gamma  \sinh 1)}{4 (\gamma + \eta  \cosh 1)^2}\right)^2\right)^2-4 
		\intertext{and}
		F\left(\frac{\pi}{2}\right) &= \left(\log \left(\cos 1 - \frac{ k(\eta  \sin 2 + 2 \gamma  \sin 1)}{4 ( \gamma +\eta \cos 1  )^{2}}\right)^2\right)^2-4.
	\end{align*}
As $\log x$ is a monotonically increasing function, it is suffices to determine the minimum of $\xi(k):=\cosh 1 +k (\eta  \sinh 2+2 \gamma  \sinh 1 )/4 (\gamma + \eta  \cosh 1)^2.$ Observe that 
 $\xi'(k)=(2 \gamma \sinh 1 +\eta \sinh 2)/4(\gamma +\eta \cosh 1)^{2}>0,$ whenever $\eta> \eta_{1},$ i.e $\xi(k)$ is an increasing function of $k,$ whenever $\eta> \eta_{1}.$ Further, as a consequence of the inequality: $\eta_{1} < \eta \leq\eta_{0}$, we have
 $\xi(k)\geq \cosh 1 +(2 \gamma \sinh 1+ \eta \sinh 2)/4(\gamma +\eta \cosh 1)^{2}=\xi(1)\geq e,$ which gives $(\log(\xi^{2}(k)))^{2} \geq (\log e^{2})^{2}=4,$ thus $F(0)\geq 0.$ Moreover, for each $k\geq 1,$ we have $\zeta(k):=\cos 1 -   k(\eta  \sin 2 + 2 \gamma  \sin 1)/(4 (\gamma + \eta  \cos 1 )^2)\geq \cos 1 -   (\eta  \sin 2 + 2 \gamma  \sin 1)/(4 (\gamma + \eta  \cos 1 )^2)=\zeta(1)\geq e,$ whenever $\eta_{2}\leq \eta < \eta_{3},$ which implies $(\log(\zeta^{2}(k)))^{2} \geq (\log e^{2})^{2}=4,$  thus $F(\pi/2)\geq 0.$ Hence $|\log \mathfrak{B}(z_{0})|^{2}\geq 1,$ which contradicts the hypothesis given in \eqref{e59}, this completes the proof. \end{proof}

\noindent  Below, we derive some special cases of Theorem \eqref{t6} by appropriately choosing the value of $\eta$ so that the conditions of the hypothesis are not violated. Next result is obtained by substituting $p(z)=zf'(z)/f(z),$ and take $\eta=1/2,$ in Theorem \ref{t6}(i) and (iii), 
\begin{corollary}
Let $\gamma\in\mathbb{R}\setminus\{-1/2\},$ and if $f\in\mathcal{A}$ satisfies the following 
\[\frac{zf'(z)}{f(z)}\left(1+\dfrac{1+2\left(\dfrac{zf''(z)}{f'(z)}-\dfrac{zf'(z)}{f(z)}\right)}{\dfrac{zf'(z)}{f(z)}+2\gamma}\right)\prec \phi(z),\] 
\begin{enumerate}[(i)]
\item for $\phi(z)=z+\sqrt{1+z^{2}},$ with 
\[-\frac{\cos 1}{2} \left(1+\frac{\tan 1}{\sqrt{2}+1-\cos 1}\right)\leq \gamma \leq -\frac{\cos 1}{2},\]
\item for $\phi(z)=e^{z},$ with \[-\frac{\cos 1}{2}  \left(1+\frac{\tan 1}{e-\cos 1}\right) \leq \gamma \leq -\frac{\cosh 1}{2} \left(1-\frac{\tanh 1}{e-\cosh 1}\right),\]
\end{enumerate}
 then $f\in\mathcal{S}^{*}_{\varrho}.$
\end{corollary}
Taking $p(z)=zf'(z)/f(z),$ $s=0.2$ and $\eta=1$ in Theorem \ref{t6}(ii), we obtain the following corollary:
\begin{corollary}
Let $\gamma\in\mathbb{R}\setminus\{-1\}$ such that $-(\sin 1)/(2((1.2)^{2}-\cos 1))\leq \gamma \leq -\cos 1.$ If $f\in\mathcal{A}$ satisfies the following 
\[\frac{zf'(z)}{f(z)}\left(1+\dfrac{1+\dfrac{zf''(z)}{f'(z)}-\dfrac{zf'(z)}{f(z)}}{\dfrac{zf'(z)}{f(z)}+\gamma}\right)\prec (1+(0.2)z)^{2},\]
then $f\in\mathcal{S}^{*}_{\varrho}.$
\end{corollary}
We obtain the following examples as a byproduct of Lemma \cite[Theorem 3.2d, p.86]{Miller & Mocanu} and Theorem \ref{t6}, for a suitable selection of different parameters. Choose $\gamma=-3/5$ and $s=1/2,$ then from Theorem \ref{t6}(ii) we have 
\begin{equation}\label{e68}
 \frac{3}{5 \cos 1} - \frac{2\sin 1}{\cos 1\left(9- 4 \cos 1\right)} \leq \eta \leq \frac{3}{5 \cosh 1}+\frac{2 \sinh 1}{\cosh 1\left(9 - 4 \cosh 1\right)}.
 \end{equation} Substitute $a=1=n,$ $\beta=\eta,$ $\gamma=-3/5$ and $h(z)=(1+(1/2)z)^{2}$ in \cite[Theorem 3.2d, p.86]{Miller & Mocanu}, then the open door function, which is univalent in $\mathbb{D},$ reduces to  \begin{equation}\label{e70}
 R_{\eta - 3/5,1}(z)=(\eta - 3/5)(1+z)/(1-z)+2 z/(1-z^{2}).
 \end{equation}


\begin{example}\label{e67}
	Let $\eta$ be given by \eqref{e68}, $\operatorname{Re}(\eta)>3/5$ and $R_{\eta - 3/5,1}(z)$ be given by \eqref{e70}. If 
	  \[\eta(1+(1/2)z)^{2}\prec R_{\eta - 3/5,1}(z)+3/5,\]  then 
	\[ p(z) =\left(\eta \int_{0}^{1} t^{\eta- 8/5} e^{(t-1)( z \eta( z(1 + t)+8))/8} dt \right)^{-1}  + \frac{3}{5 \eta} \]
is analytic in $\mathbb{D},$ and it is a solution of the differential equation $p(z)+z p'(z)/(\eta p(z) -3/5) = (1+ (1/2)z)^{2}$  and satisfies $\operatorname{Re}(\eta p(z))>3/5.$ Furthermore,  $p(z)\prec \cosh \sqrt{z}.$ 
\end{example}
\noindent By taking $\eta=1/2,$ in Theorem \ref{t6}(iii), we deduce that 
\begin{equation}\label{e69}
-\frac{1}{2} \left(\cos 1+\frac{\sin 1}{e-\cos 1}\right)\leq \gamma \leq -\frac{1}{2} \left(\cosh 1-\frac{\sinh 1}{e-\cosh 1}\right).
 \end{equation} Choose $n=a=1,\beta=\eta=1/2$  and $h(z)=e^{z}$ in 
	\cite[Theorem 3.2d, p.86]{Miller & Mocanu}, then the open door function, which is univalent in $\mathbb{D},$ becomes \begin{equation}\label{e71}
 R_{\gamma +1/2,1} (z)=(\gamma +1/2)(1+z)/(1-z) + 2z/(1-z^{2}).
 \end{equation} 
\begin{example}\label{ex1}
Let $\gamma$ be given by \eqref{e69},   
$\operatorname{Re}\gamma>-1/2$ and $R_{\gamma +1/2,1}(z)$ be given by \eqref{e71}. If 
\[\gamma + e^{z}/2 \prec  R_{\gamma +1/2,1} (z), \] then 
	\[p(z)=\left(\frac{1}{2}\int_{0}^{1}t^{\gamma-1} \left(e^{Chi(tz)-Chi(z)+Shi(tz)-Shi(z)} \right)^{1/2}  dt \right)^{-1}-2 \gamma, \] 
	is analytic in $\mathbb{D},$ and it is a solution of the differential equation  $p(z)+2z p'(z)/(p(z) +2\gamma)=e^{z},$ 
 where 
	$Chi({z}):=\xi+\log z+\displaystyle{\int_{0}^{z}(\cosh{t}-1)/t}$ $dt$  and  $Shi(z)=\displaystyle{\int_{0}^{z}\sinh t /t}$ $dt$, and satisfies $\operatorname{Re}p(z)>-2\gamma.$ Then $p(z)\prec \cosh\sqrt{z}.$ 
\end{example}

\begin{theorem}\label{t9}
	Let $-1\leq B< A \leq 1$ and $\eta,\gamma\in\mathbb{R}$ such that $\gamma\neq -\eta,$ satisfy the following conditions:
	
	\begin{enumerate}[(i)]
		\item $(1-B^{2}) \sinh 1 + 2(\gamma + \eta \cosh 1)(\cosh 1 - 1 + B(A - B \cosh 1)) \geq 0$
		\item $(\sinh 1 + 2 (\cosh 1 -1)(\gamma +\eta \cosh 1))^{2} \geq (B \sinh 1 - 2(A - B \cosh 1)(\gamma + \eta \cosh 1))^{2}$
		\item $(1-B^{2}) \sin 1 + 2(\gamma +\eta \cos 1)(1 - \cos 1 - B(A - B \cos 1))\geq 0$
		\item $(\sin 1 + 2(\cos 1 - 1)(\gamma + \eta \cos 1))^{2} \geq (B \sin 1+2(A - B \cos 1)(\gamma + \eta \cos 1))^{2}.$
	\end{enumerate}
	Let $p(z)$ be an analytic function, such that  $p(0)=1$ and satisfies
	\begin{equation*} 
		p(z)+\frac{zp'(z)}{\eta p(z) + \gamma} \prec \frac{1+A z}{1 + B z},
	\end{equation*}
	then
	$p(z)\prec \cosh \sqrt{z}.$ 
\end{theorem}
\noindent The proof of Theorem \ref{t9} is much akin to the previous results, so it is omitted.

The following corollaries illustrate specific outcomes of Theorem \ref{t9}, derived by substituting $p(z) = zf'(z)/f(z)$ and setting the parameters as follows: $A = 1$, $B = 0$, $\gamma = 0$; and $A = 0$, $B = -1/2$, $\eta = 1$, respectively.
\begin{corollary}\label{e73}
Let $\eta\in\mathbb{R}\setminus\{0\}$ such that   
$-(\tanh 1\sech 1)/2\leq \eta \leq (\tanh 1)/(4-2\cosh 1).$ If  $f\in\mathcal{A}$ satisfies 
\begin{equation}\label{e74}
1+\frac{zf''(z)}{(1-\eta)f'(z)}-\frac{zf'(z)}{f(z)}\prec \frac{\eta z}{1-\eta},
\end{equation}
then $f\in\mathcal{S}^{*}_{\varrho}.$
\end{corollary}
\begin{corollary}
Let $\gamma \in\mathbb{R}\setminus\{-1\}$ such that
$(\sin 1+4 \cos 1-\cos 2-1)/(2 \cos 1-4)\leq \gamma \leq (4 \cosh 1-\sinh 1-2 \cosh ^2 1)(2 \cosh 1-4).$ If $f\in\mathcal{A}$ satisfies 
\[\frac{zf'(z)}{f(z)}\left(1+\dfrac{1+\dfrac{zf''(z)}{f'(z)}-\dfrac{zf'(z)}{f(z)}}{\dfrac{zf'(z)}{f(z)}+\gamma}\right)\prec \frac{2}{2-z},\] then $f\in\mathcal{S}^{*}_{\varrho}.$
\end{corollary}


Note that  \eqref{e74} in Corollary \ref{e73} is equivalent to: 
\[\left|1+\frac{zf''(z)}{(1-\eta)f'(z)}-\frac{zf'(z)}{f(z)}\right| < \left|\frac{\eta}{1 - \eta }\right|.\] 

It is observed that for $g\in\mathcal{A},$ the Briot-Bouquet differential equation is closely related to the  Bernardi integral operator \cite{Miller & Mocanu}, given by
\begin{equation}\label{e37}
G(z):=\frac{\eta +\gamma}{z^{\gamma}}\int_{0}^{z} g^{\eta}(t) t^{\gamma-1} dt.
\end{equation}
If $p(z)=zG'(z)/G(z),$ then we have the following Briot-Bouquet differential equation:
\begin{equation}\label{e72}
p(z)+\frac{zp'(z)}{\eta p(z)+\gamma}=\frac{zg'(z)}{g(z)}.
\end{equation}
Using this fact, we derive the following corollary, as a consequence of Theorem \ref{t6} and \ref{t9} :

\begin{corollary}\label{c2}
Let $-1\leq B< A \leq 1$ and $\eta,\gamma\in\mathbb{R}$ such that $\gamma\neq -\eta.$  Assume the  conditions as outlined in Theorem \ref{t6} \eqref{e29}-\eqref{e28} and Theorem \ref{t9}(i)-(iv) holds.  If  $g\in\mathcal{S}^{*}(\phi),$ then $G\in\mathcal{S}^{*}_{\varrho}$ for the choices of $\phi(z):$ $z+\sqrt{1+z^{2}},$ $(1+sz)^{2},$ $e^{z}$ and $(1+Az)/(1+Bz)$  
respectively.
\end{corollary}
\begin{proof}
Let us first prove the result for the case when $\phi(z)=z+\sqrt{1+z^{2}},$ and other cases will follow in the similar fashion.	Assume that $p(z)=z G'(z)/G(z),$ where $G(z)$ is given by \eqref{e37}. Since  
$g\in\mathcal{S}^{*}(\phi),$ then from \eqref{e72}, we have  \[p(z)+\frac{zp'(z)}{\eta p(z)+\gamma}= \frac{zg'(z)}{g(z)}\prec z+\sqrt{1+z^{2}}.\] Now the result follows at once by an  application of Theorem \ref{t6}(i).
\end{proof}



\section{First order differential subordination Results}
In this section, we study certain differential subordination implications result involving the expressions: $1+\eta zp'(z)/p(z),$ $1+\eta z p'(z)$ and  $\eta p(z)+zp'(z)/p(z).$  
\noindent We require the following Miller and Mocanu's Lemma to derive our results. 
\begingroup
\setcounter{tmp}{\value{theo}}
\setcounter{theo}{1} 
\renewcommand\thetheo{\Alph{theo}}
\begin{theo}\cite{Miller & Mocanu} \label{l7}
Let	$q$ be analytic in $\mathbb{D}$ and let $\varphi$ be analytic in domain $\mathcal{D}$ containing $q(\mathbb{D})$ with $\varphi(w)\neq 0$ when $w\in q(\mathbb{D}).$ Set $\mathcal{Q}(z):=zq'(z)\varphi(q(z))$ and $h(z):=\nu(q(z))+\mathcal{Q}(z).$ Suppose (i.) either $h$ is convex, or $\mathcal{Q}$ is starlike univalent in $\mathbb{D}$ and (ii.) $\operatorname{Re}(zh'(z)/\mathcal{Q}(z))>0 $ for $z\in\mathbb{D}.$ If $p$ is analytic in $\mathbb{D},$ with $p(0)=q(0),$ $p(\mathbb{D})\subset \mathcal{D}$ and 
\[\nu(p(z))+zp'(z)\varphi(p(z))\prec \nu(q(z))+zq'(z)\varphi(q(z)),\]
then $p(z)\prec q(z),$ and $q$ is the best dominant.
\end{theo}
\endgroup
\noindent We begin with the following result.
\begin{theorem}\label{t11}
	Suppose   $A,B\in\mathbb{C},$ with $A \neq B$ and $|B|<1,$ $\eta$ be such that \[|\eta|\geq  \dfrac{2|A-B|}{(1-|B|)\tanh 1}.\]   Let $p(z)$ be analytic in $\mathbb{D}$ with $p(0)=1,$ satisfying the subordination \[1+\eta \frac{z p'(z)}{p(z)} \prec \frac{1+A z}{1+B z},\]  then  $p(z)\prec \cosh \sqrt{z}$ and $\cosh \sqrt{z}$ is the best dominant.
\end{theorem}

\begin{proof}
	 Let   $q(z)=\cosh \sqrt{z},$ $ \nu (w)=1$ and $\varphi(w)= \eta / w,$ then $\mathcal{Q}(z)=\eta z q'(z)/q(z)$ $=\eta (\sqrt{z} \tanh{\sqrt{z}})/2,$ which implies  $\operatorname{Re}z \mathcal{Q}'(z)/\mathcal{Q}(z)> 0$ for each $z\in\mathbb{D}.$ Hence $\mathcal{Q}(z)$ is starlike in $\mathbb{D}.$ Since $h(z)= \nu(\varrho(z)) + \mathcal{Q}(z),$ we have \[\operatorname{Re} \left(\frac{z h'(z)}{\mathcal{Q}(z)} \right) =\frac{1}{2} + \operatorname{Re} \sqrt{z} \csch 2 \sqrt{z} > \frac{1}{2}+\csch 2 >0.\]  
	Assume $\Psi(z)= (1+Az)/(1+Bz),$ then from the representation of $\Psi(z),$ we can write $\Psi^{-1}(w)=(w-1)/(A-B w).$
	If we choose \[T(z)= 1+\frac{\eta}{2} \sqrt{z} \tanh(\sqrt{z}),\] then it suffices to show that $\Psi(z)\prec T(z),$ for each $z\in\mathbb{D}$ or $\mathbb{D}\subset \Psi^{-1}(T(\mathbb{D}))$ or equivalently,  
	\[|\Psi^{-1}(T(e^{i t}))|\geq 1 \quad (-\pi\leq t\leq \pi).\]
  Since $|\tanh(e^{i t/2})|$
	attains its minimum at $t=0$ and  for $|\eta|\geq 2|A-B|/(1-|B|)\tanh 1,$ on the boundary of $\mathbb{D},$ we have  
	\begin{align*}
		|\Psi^{-1}(T(e^{i t}))| &
		\geq \frac{|\eta||\tanh(e^{i t/2})|}{2 |A-B| + |\eta B \tanh(e^{i t/2})|}\\&
  \geq \frac{|\eta|\tanh 1}{2 |A-B| + |\eta B| \tanh 1}\\&
  \geq 1.
	\end{align*}
Now the result follows at once by an application of  Lemma \ref{l7}.      
\end{proof}

\noindent By taking $p(z)=zf'(z)/f(z)$ in Theorem \ref{t11}, we obtain the following corollary.
\begin{corollary}
	Assume $\eta,$ $A$ and $B$ as given in Theorem \ref{t11}.  Let $f\in\mathcal{A}$ such that 
	\[1+\eta \left(1+\frac{z f''(z)}{f'(z)} - \frac{zf'(z)}{f(z)} \right) \prec \frac{1+A z}{1 + B z},\] then $f\in\mathcal{S}^{*}_{\varrho}.$
\end{corollary}

In the next result, we establish sharp lower bound on $\eta,$ for which the following implication holds:
\[1+\eta z p'(z) \prec \cosh \sqrt{z} \Rightarrow p(z)\prec \phi(z),\]
where we choose $\phi(z)$ to be any of $\phi_{e}(z),\phi_{L}(z),\phi_{\rightmoon}(z),\phi_{A,B}(z)$ and $\phi_{s}(z)$.

\begin{theorem}\label{t12}
	Let $\eta\in\mathbb{R}$ and $p(z)$ be analytic in $\mathbb{D}$ satisfying $1+\eta z p'(z)\prec \cosh \sqrt{z},$             	
 then 
	\begin{enumerate}[(i)]
		\item $p(z)\prec e^{z}$ and $(e-1)\eta \geq 2 e (\xi - Ci(1)).$ 	
  \item $p(z)\prec \sqrt{1+z}$  and $(\sqrt{2}-1)\eta \geq 2(Ci(1)-\xi).$
		\item $p(z)\prec z+\sqrt{1+z^{2}}$ and $(\sqrt{2}-1)\eta \geq \sqrt{2}(\xi - Ci(1)).$ 
		\item $p(z) \prec (1+A z)/(1+B z)$ and $\eta \geq \eta_{A,B},$ where  
		
		\begin{align}\label{e52}
			\eta_{A,B}& =  
			\begin{cases} 
				\dfrac{2}{A-B}(1-B)(\xi - Ci(1))& \text{if } -1 < B \leq -\kappa,  \\ & \\
				\dfrac{-2}{A-B}(1+B)(\xi - Chi(1)) & \text{if }   -\kappa < B < 1
			\end{cases}
		\end{align}
		and $-1<B < A \leq 1.$ 
		
		\item $p(z)\prec (1+ s z)^{2}$ and $\eta\geq \eta_{s},$ where 
		\begin{align}\label{e54}
			\eta_{s}& =  
			\begin{cases} 
				\dfrac{2}{s(s+2)}( Chi(1) - \xi ) & \text{if } 0 < s \leq 2\kappa,  \\ & \\
				\dfrac{2}{s(s-2)}( Ci(1) - \xi ) & \text{if }   2 \kappa < s \leq 1/\sqrt{2}.
			\end{cases}
		\end{align}
		
	\end{enumerate} 
 where
\begin{equation}\label{e47}
\kappa:=\frac{Chi(1) + Ci(1)- 2\xi}{Chi(1)-Ci(1)}, 
\end{equation} 
	with \begin{equation}\label{e48}
 Chi(1) = \xi +\int_{0}^{1}(\cosh \sqrt{t} -1)/t  \hskip 0.2cm dt \quad \text{and} \quad Ci(1) = \xi +\int_{0}^{1}(\cos \sqrt{t} -1)/t \hskip 0.2cm dt. 
 \end{equation}
All the bounds on $\eta$ are sharp.
\end{theorem}

\begin{proof}
The differential equation $1+\eta z \phi_{\eta}'(z)= \cosh\sqrt{z},$ has a solution $\phi_{\eta}(z):\overline{\mathbb{D}} \rightarrow \mathbb{C},$ given by   
		\[\phi_{\eta}(z)=1+\frac{1}{\eta }\left(2 Chi\sqrt{z} - \log z - 2 \xi \right),\] where $\xi \approx 0.577216$ is the Euler–Mascheroni constant.
	 Let $\nu (w)=1$ and $\varphi(w)=\eta.$ Define an analytic function $\mathcal{Q}:\overline{\mathbb{D}}\rightarrow \mathbb{C}$ by 
	\[\mathcal{Q}(z)=z \phi'_{\eta}(z)\varphi(\phi_{\eta}(z))=-1+\cosh\sqrt{z}. \]
 
	Since $\operatorname{Re} z\mathcal{Q}'(z)/\mathcal{Q}(z)>(1/2)\cot (1/2)>0,$  thus $\mathcal{Q}(z)$ is starlike in $\mathbb{D}.$ Let $h(z)=\nu(\phi_{\eta}(z))+\mathcal{Q}(z),$ then we have
	\[\operatorname{Re}\left(\frac{zh'(z)}{\mathcal{Q}(z)}\right)=\operatorname{Re}\left(\frac{1}{2} \sqrt{z} \coth \left(\frac{\sqrt{z}}{2}\right)\right)>0.\] 
By an application of Lemma \ref{l7}, we obtain $1+\eta z p'(z)\prec 1+\eta z \phi_{\eta}'(z),$ which implies $p(z)\prec \phi_{\eta}(z).$ Now we need to show that $\phi_{\eta}(z)\prec \phi(z),$ where $\phi(z)$ is any of these functions: $\phi_{e}(z),$ $\phi_{L}(z)$,$\phi_{\rightmoon}(z),$ $\phi_{A,B}(z)$ and $\phi_{s}(z).$ 
If $\phi_{\eta}(z)\prec \phi(z),$ then 
\begin{equation}\label{e53}
		\phi(-1) \leq \phi_{\eta}(-1)<\phi_{\eta}(1) \leq \phi(1).
  \end{equation}
Now, for each choice of $\phi(z)$, by solving equation \eqref{e53}, we obtain sharp bounds on $\eta$. The graphical observations presented in Figure \ref{smallsubdn} demonstrate that the condition in \eqref{e53} is not only necessary but also sufficient for the chosen choice of $\phi(z)$.

\begin{figure}[!ht]
    \centering
\subfloat[\centering]{\includegraphics[height=1.55in,width=1.65in]{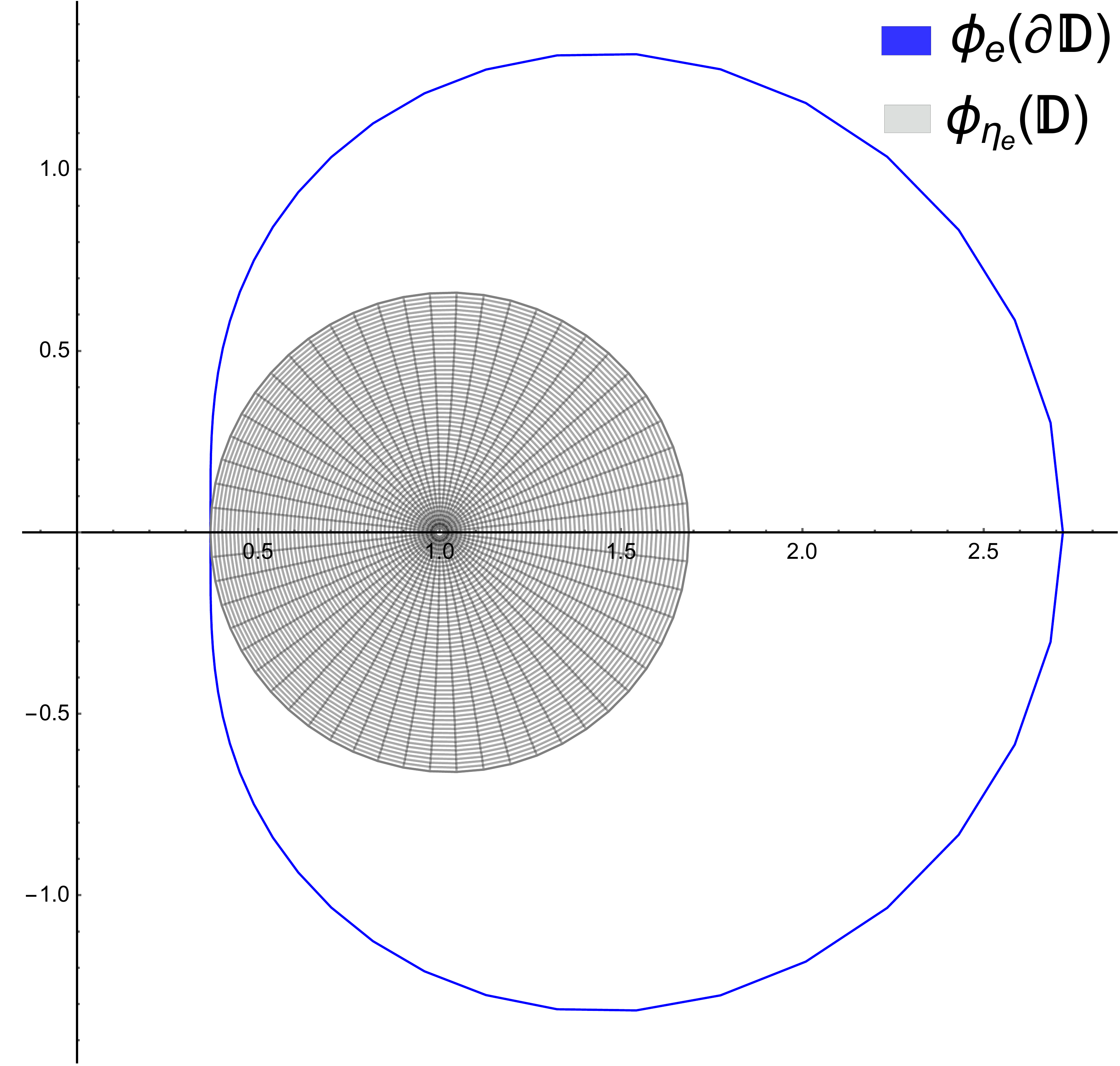}}  \qquad \qquad \qquad
\subfloat[\centering]{\includegraphics[height=1.55in,width=1.45in]{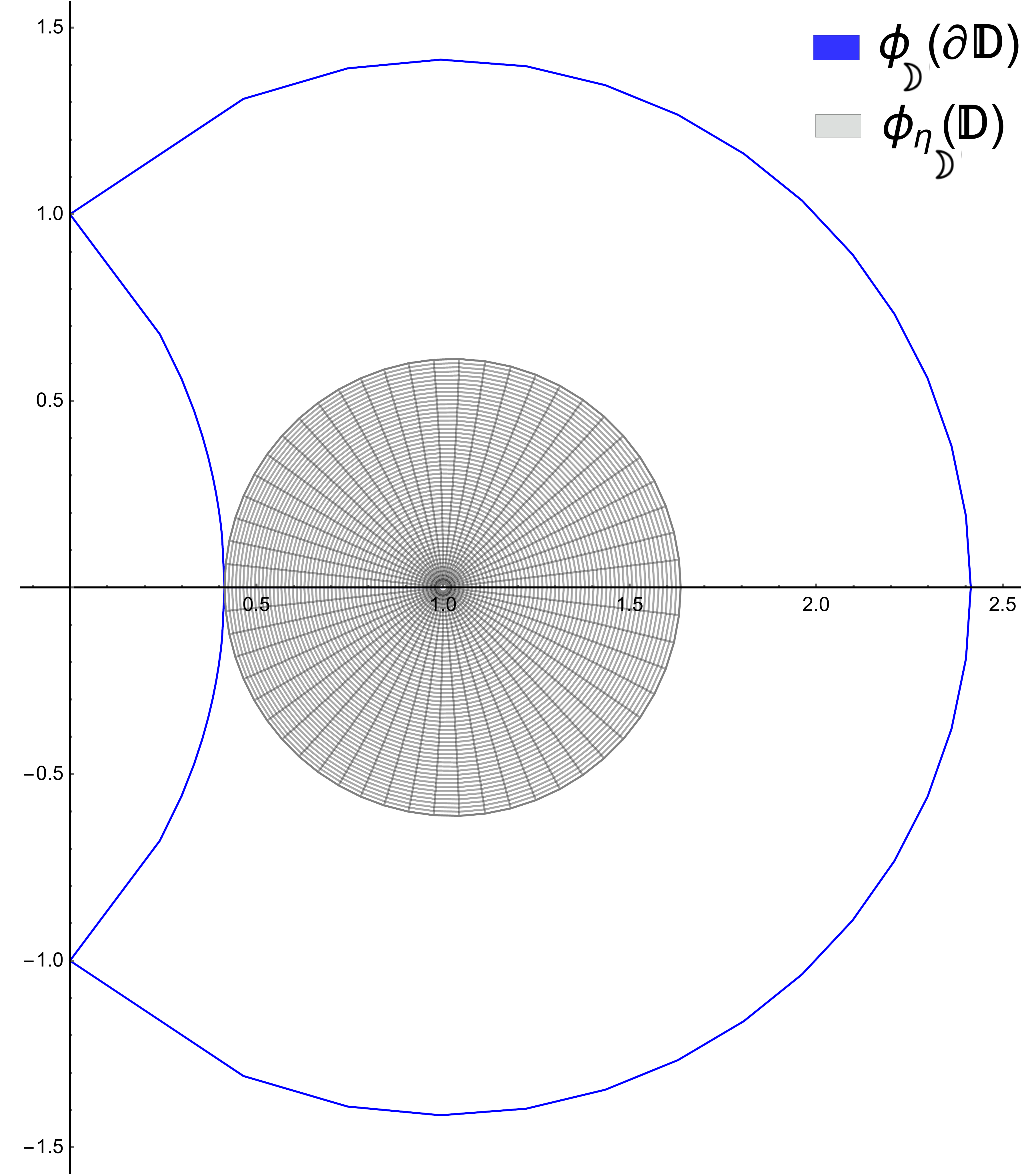}}  \qquad \qquad \qquad 
\subfloat[\centering]
{\includegraphics[height=1.25in,width=1.75in]{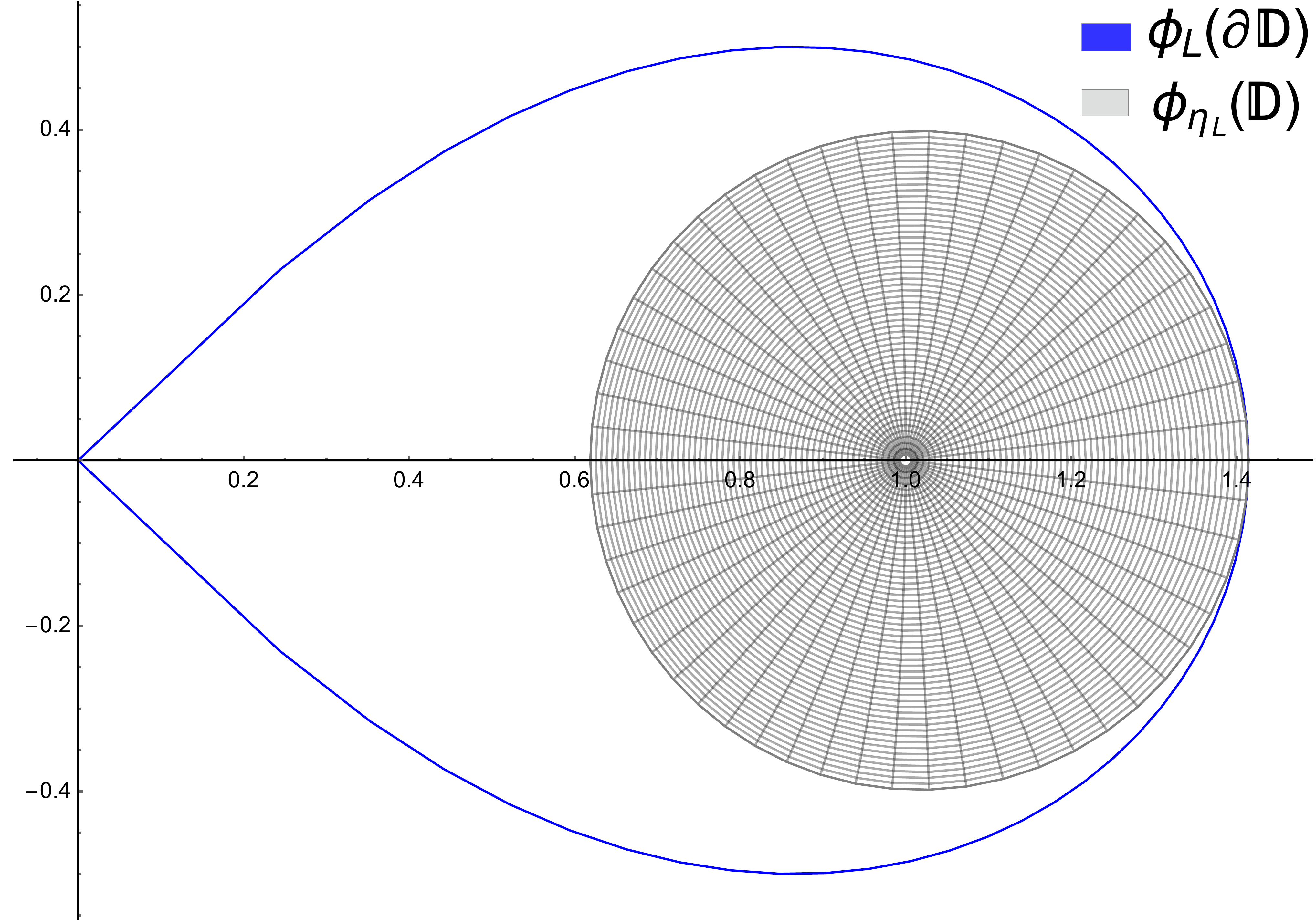}}  \qquad \qquad \qquad
\subfloat[\centering]{\includegraphics[height=1.8in,width=1.8in]{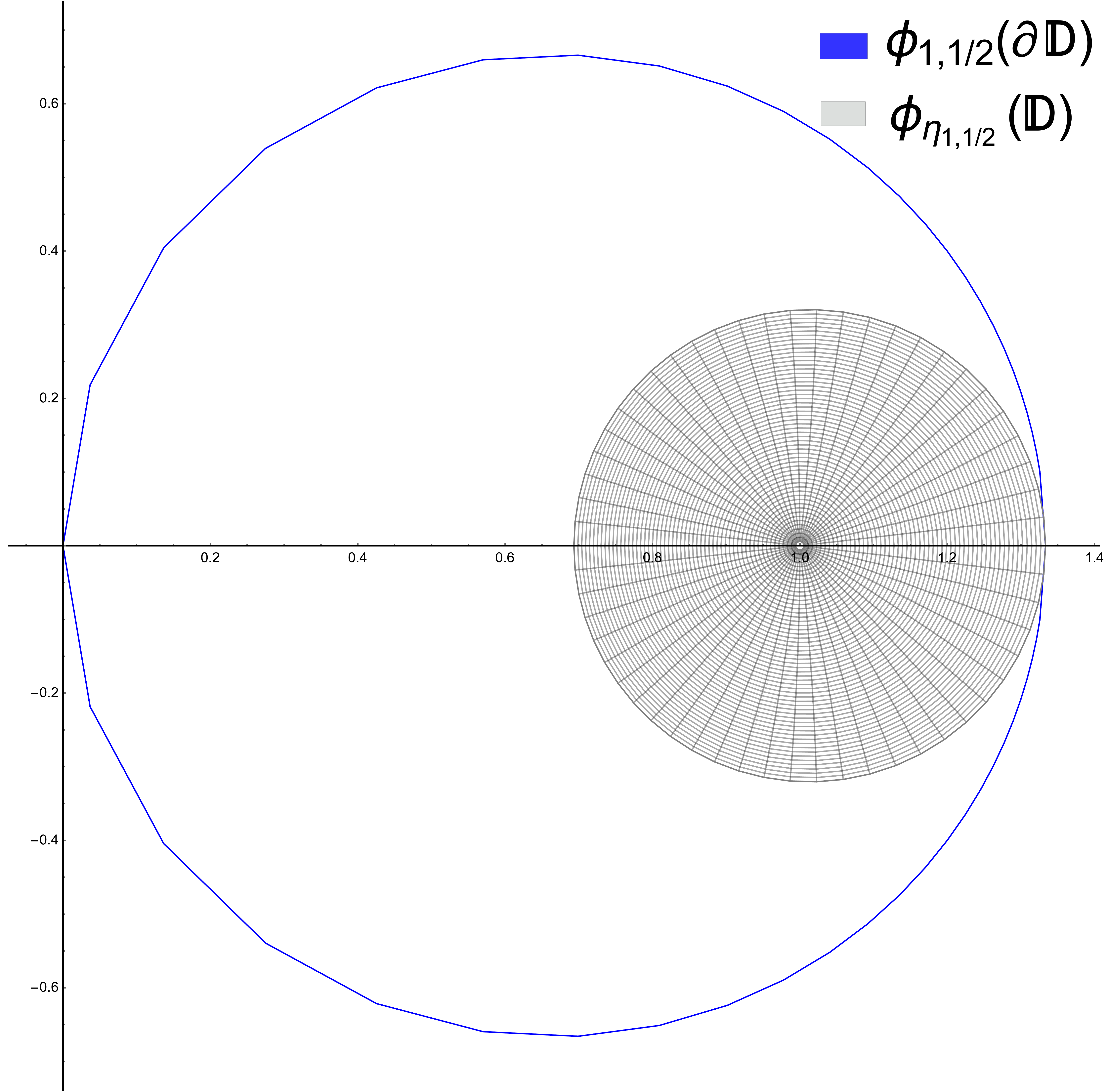}} \qquad \qquad 
\subfloat[\centering]{\includegraphics[height=1.7in,width=1.5in]{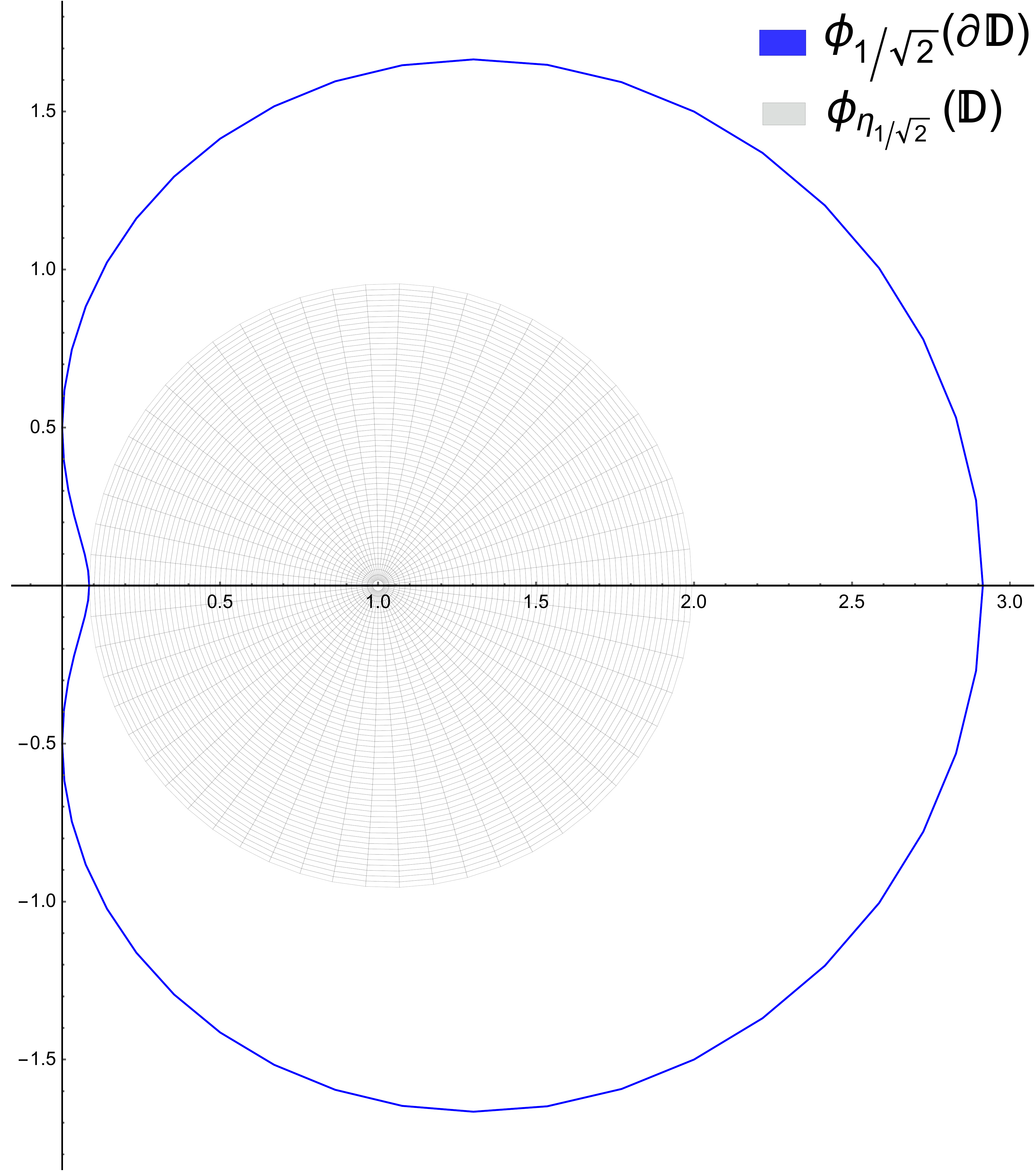}}
\caption{Images of $\partial\mathbb{D}$  and $\mathbb{D}$ under the mappings,  (A) $\phi_{e}(z)$ and $\phi_{\eta_{e}}(z),$  (B) $\phi_{\rightmoon}(z)$ and $\phi_{\eta_{\rightmoon}}(z),$ (C) $\phi_{L}(z)$ and $\phi_{\eta_{L}}(z),$   (D) $\phi_{1,1/2}(z)$ and $\phi_{\eta_{1,1/2}}(z)$ \& (E)  $\phi_{1/\sqrt{2}}(z)$ and $\phi_{\eta_{1/\sqrt{2}}}(z),$ respectively.}
\label{smallsubdn}
\end{figure}

\begin{enumerate}[(i)]		
\item When $\phi(z)=\phi_{e}(z),$ then \eqref{e53} reduces to the following  
		\begin{equation}\label{e50}
			e^{-1} \leq 1 + \frac{2(Ci(1) - \xi)}{\eta} < 1 + \frac{2(Chi(1) - \xi)}{\eta} \leq e.	
		\end{equation}
  Now from \eqref{e50} we get 
  \[e^{-1}\leq 1 + \frac{2(Ci(1) - \xi)}{\eta} ,\text{ which implies } \eta \geq \frac{2 e (\xi - Ci(1))}{e-1}=\eta_{e}\approx 0.758753\]
  and \[1 + \frac{2(Chi(1) - \xi)}{\eta} \leq e,\text{ which implies } \eta \geq \frac{2 (Chi(1)-\xi)}{e-1}=x_{e}\approx 0.303386.\]
  It can be observed that $\eta \geq \max\{\eta_{e},x_{e}\}.$ Therefore, $\phi_{\eta}(z)\prec \phi(z)$ for each $\eta\geq \eta_{e},$ consequently, due to transitivity the result follows at once.

\item If $\phi(z)=\phi_{L}(z),$ then inequalities \eqref{e53} leads to  \[0\leq \frac{2( Ci(1)- \xi)}{\eta }+1<\frac{2 (Chi(1)-\xi)}{\eta }+1\leq \sqrt{2},\] which is true, provided $\eta\geq \eta_{L},$ where \[\eta_{L}=\frac{\sqrt{2}(\xi - Ci(1))}{\sqrt{2}-1}\approx 1.25854.\] Thus  for each $\eta\geq \eta_{L},$ we have $p(z)\prec \phi_{L}(z).$
\item For $\phi(z)=\phi_{\rightmoon}(z),$ the expression in \eqref{e53}, gives 
	\begin{equation}\label{e51}
			\sqrt{2} -1 \leq 1 + \frac{2(Ci(1) - \xi)}{\eta} < 1 + \frac{2(Chi (1) - \xi)}{\eta} \leq \sqrt{2}+1,	
		\end{equation}   
		which holds for each $\eta \geq \eta_{\rightmoon},$ where  \[\eta_{\rightmoon}= \frac{\sqrt{2}(\xi - Ci(1))}{ \sqrt{2}-1}\approx 0.818769.\]
		Therefore, for each $\eta \geq \eta_{\rightmoon},$ we have $\phi_{\eta}(z)\prec \phi_{\rightmoon}(z).$ Accordingly we conclude that $p(z)\prec \phi_{\rightmoon}(z).$
  \end{enumerate}
The cases (iv) and (v) will follow similar to the above cases.  Sharpness of the result is depicted in Figure \ref{smallsubdn}, for the above choices of $\phi(z).$ We 
choose $A=1,B=1/2$ and $s=1/\sqrt{2},$ in parts (iv) and (v), respectively for graphical representation of the case, consequently, we obtain  $\eta\geq \eta_{1/\sqrt{2}}\approx 0.52463$ and $\eta\geq \eta_{1,1/2}\approx 1.56391,$ respectively, from \eqref{e52} and \eqref{e54}.\end{proof}

\begin{corollary}
Let $\kappa,$ $Chi(1)$ and $Ci(1)$  be as given in Theorem \eqref{t12}. For $f\in\mathcal{A},$ if
\[\Phi_{\eta}(z):=1+\eta  \frac{z f'(z)}{f(z)}\left(1-\frac{z f'(z)}{f(z)}+\frac{zf''(z)}{f'(z)}\right),\] then 
$f\in\mathcal{S}^{*}_{\varrho}$ implies the following 
\begin{enumerate}[(i)]
\item $\Phi_{\eta}(z)\prec e^{z},$ whenever $(e-1)\eta\geq 2 e(\xi - Ci(1)).$
\item $\Phi_{\eta}(z)\prec \sqrt{1+z},$ whenever $(\sqrt{2}-1)\eta \geq 2 (Ci(1)-\xi).$
\item $\Phi_{\eta}(z)\prec z+\sqrt{1+z^{2}},$ whenever $(\sqrt{2}-1)\eta \geq \sqrt{2}(\xi - Ci(1)).$
\item $\Phi_{\eta}(z)\prec (1+Az)/(1+Bz),$ whenever $\eta \geq \eta_{A,B},$ where $-1<B<A\leq 1$ and $\eta_{A,B}$ is given in  Theorem \ref{t12}.
\item $\Phi_{\eta}(z)\prec (1+sz)^{2},$ whenever $\eta \geq \eta_{s},$ where $\eta_{s}$ is given in Theorem \ref{t12}.
\end{enumerate}
All the bounds obtained on $\eta$ are sharp.
\end{corollary}

\begin{theorem}\label{t13}
	Let \( -1 < B < A \leq 1 \), \(\mu = (\cosh 1 - 1)/(1 - \cos 1)\), and \( B_{0} \) be the root of the equation \((1 + B)(1 - B)^{\mu} = 1\). If \( p(z) \) is an analytic function in \(\mathbb{D}\) that satisfies \( 1 + \eta z p'(z) \prec (1 + Az)/(1 + Bz) \), then \( p(z) \prec \cosh \sqrt{z} \),
 provided $\eta$ satisfies the following sharp inequalities
\begin{align}\label{e55}
		& 
  \left.
		\begin{array}{lll}
			\eta  B (\cosh 1 - 1) \leq (A - B) \log (1 + B)  & \text{ if }  -1 < B \leq B_{0}, \\ 
			\eta B (\cos 1 - 1)  \geq (A - B) \log (1 - B)  & \text{ if }   B_{0} < B < 0,\\
			\eta B (\cos 1 - 1)  \leq (A - B) \log (1 - B)   & \text{ if }  0 < B < 1, \\
   2\eta \geq A \csc^{2}(1/2) &  \text{ if } B=0.
\end{array}
  \right\}
	\end{align} 
\end{theorem}

\begin{proof}
The analytic function   
	$\phi_{\eta}:\overline{\mathbb{D}} \rightarrow \mathbb{C}$ defined as : 
	\[\phi_{\eta}(z)=1+\frac{A-B}{\eta B} \log(1+Bz),\]
 is a solution of the differential equation $1+\eta z \phi_{\eta}'(z)= (1+Az)/(1+Bz).$ We shall prove the result by using Lemma \ref{l7}, accordingly we assume
$\nu(u)=1,$  $\varphi(u)=\eta$ and define an analytic function $\mathcal{Q}:\overline{\mathbb{D}}\rightarrow \mathbb{C}$ as 
	\[\mathcal{Q}(z)=z \phi_{\eta}'(z)\varphi(\phi_{\eta}(z))=(A-B)z/(1+Bz).\]  Clearly, for the given choice of $A$ and $B,$ $\mathcal{Q}(z)$ is a starlike function in $\mathbb{D}.$ Note that if   $h(z)=\nu(\phi_{\eta}(z))+\mathcal{Q}(z),$ then 
	$z h'(z)/\mathcal{Q}(z) = 1/(1+B z),$ thus $\operatorname{Re}z h'(z)/\mathcal{Q}(z)>1/(1+|B|)>0.$ Further, in view of Lemma \ref{l7}, we conclude that $p(z) \prec \phi_{\eta}(z).$ Now we need to show that $\phi_{\eta}(z)\prec \cosh \sqrt{z}=:\varrho(z).$  We know that the following condition is necessary:
\begin{equation}\label{e30}
\varrho(-1) \leq \phi_{\eta}(-1) < \phi_{\eta}(1) \leq \varrho(1),
\end{equation}
for $\phi_{\eta}(z) \prec \varrho(z)$. However, a graphical observation presented in Figure \ref{graph1complex} for the value of $\eta$ satisfying \eqref{e30} shows that \eqref{e30} is not only necessary but also sufficient.
For $B \neq 0,$ we have from \eqref{e30}:
\begin{equation}\label{e38}
\cos 1 \leq 1+\frac{A-B}{\eta B}\log(1-B)  < 1+\frac{A-B}{\eta B}\log(1+B) \leq \cosh 1,
\end{equation} 
\begin{figure}[H]
    \centering
{\includegraphics[height=1.6in,width=2.8in]{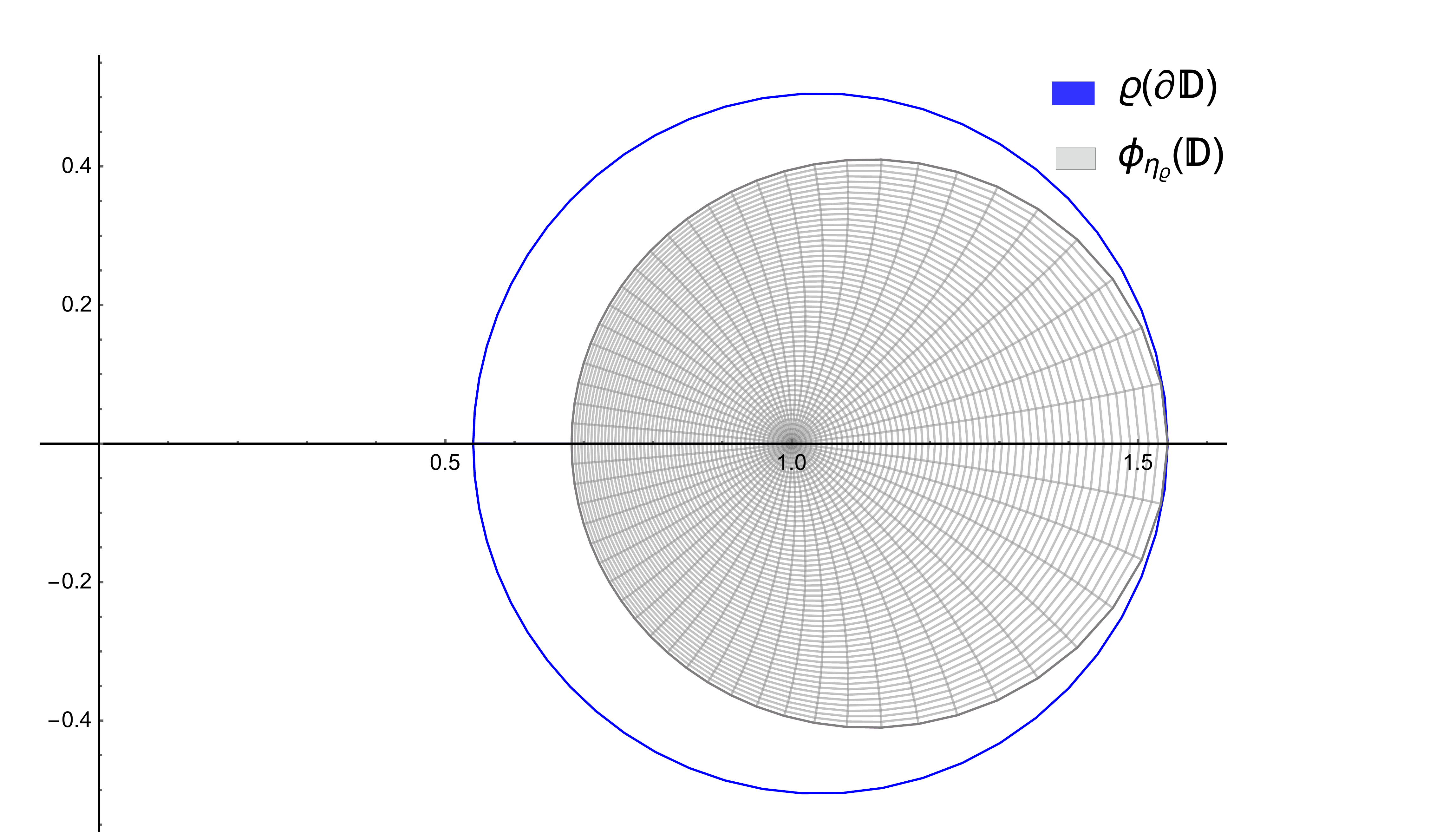}}   
\caption{Images of $\partial\mathbb{D}$ and $\mathbb{D}$ under the mappings $\varrho(z)=\cosh \sqrt{z}$  and $\phi_{\eta_{\varrho}}(z)$  respectively.}
\label{graph1complex}
\end{figure} 
\noindent  
then from \eqref{e38}, we conclude that $\phi_{\eta}(z)\prec \varrho(z),$ provided $\eta$ satisfies \eqref{e55}. 
Furthermore, if $B = 0,$ the function 
\[m_{\eta}(z):=1+(A/\eta)z\]
is a solution of the differential equation: $1+\eta z m_{\eta}'(z)=1+Az.$ Now in view of \eqref{e30}, we have the following inequality 
\begin{equation}\label{e60}
	    \cos 1\leq m_{\eta}(-1) < m_{\eta}(1) \leq \cosh 1,
\end{equation} which holds, whenever $A \leq 2 \eta {\sin}^{2}(1/2).$ Therefore, for $B=0$ we have $\phi_{\eta}(z) \prec \varrho(z),$ provided $\eta$ satisfies \eqref{e60}. Hence the result follows at once.  
\end{proof}
Note that Figure \ref{graph1complex}, illustrates  sharpness on $\eta$ for the case when  $A=1/2,$ $B=-1/2$ and $\eta=\eta_{\varrho}:=(\log 4)/(\cosh 1 -1)\approx 2.5526$ in Theorem \ref{t13}, where the extremal function is  $\phi_{\eta_{\varrho}}(z):=1-(2/\eta_{\varrho})\log(1-(z/2)).$

In Theorem \ref{t13}, choose $p(z)=zf'(z)/f(z),$ then we deduce the next result:
\begin{corollary}
Assume $A,B$ and $\mu$ as given in Theorem \ref{t13} and for $f\in\mathcal{A},$ 
\[\Phi_{\eta}(z):=1+\eta  \frac{z f'(z)}{f(z)}\left(1-\frac{z f'(z)}{f(z)}+\frac{zf''(z)}{f'(z)}\right).\]  If $\Phi_{\eta}(z)\prec (1+Az)/(1+Bz),$ then $f\in\mathcal{S}^{*}_{\varrho},$ provided $\eta$ satisfies \eqref{e55}. All the bounds attained on $\eta$ are sharp. 
\end{corollary}

\noindent In the next result, we apply Lemma \ref{l7} and derive its corresponding corollaries.
\begin{theorem}\label{t14}
Let $p(z)$ be a non-vanishing analytic function in $\mathbb{D}$ with $p(0)=1,$ such that  
\[\eta \text{} p(z) + \frac{zp'(z)}{p(z)}\prec \eta \cosh \sqrt{z}+\frac{\sqrt{z}}{2}\tanh \sqrt{z},\] where 
 $\eta \geq \eta_{0}$ with  
\begin{align}\label{e39}
\eta_{0}=-\left(\frac{1}{2}+\csch 2 \right)\sech 1\approx -0.502\cdots.
\end{align}
Then $p(z)\prec \cosh \sqrt{z},$ and $\cosh \sqrt{z}$ is the best dominant.
\end{theorem}
\begin{proof}  
Let $q(z)=\cosh \sqrt{z} $ and $\varphi(w)=1/w.$ Clearly, $q(z)$ is a convex univalent function with $q(0)=1$ and 
\[\mathcal{Q}(z)=zq'(z)\varphi(q(z))=\frac{zq'(z)}{q(z)}=\frac{\sqrt{z}}{2}\tanh \sqrt{z}.\] Now, it can be easily verified that $\operatorname{Re}(z \mathcal{Q}'(z)/\mathcal{Q}(z))=\operatorname{Re}((1/2)+\sqrt{z} \csch 2 \sqrt{z})>(1/2)+\csch 2 > 0,$
then $\mathcal{Q}(z)$ is starlike in $\mathbb{D}.$ Further, $\varphi(z)$ is analytic in $\mathbb{C}\setminus\{0\}$ containing $ q(\mathbb{D})$ with $ \varphi(w)\neq 0,$ where $w\in q(\mathbb{D}).$ Now set \[\nu(w)=\eta w \text{ } \text{ and } \text{ } h(z)=\nu(q(z)) + \mathcal{Q}(z) = \eta q(z) +\frac{zq'(z)}{q(z)} = \eta \cosh \sqrt{z}+\frac{\sqrt{z}}{2}\tanh \sqrt{z}.\]
For $\eta\geq \eta_{0},$ where $\eta_{0}$ is given by \eqref{e39}, it can be verified that
\[\operatorname{Re} \left(\eta \cosh \sqrt{z} + \sqrt{z}\csch 2 \sqrt{z}\right)>-\frac{1}{2},\] therefore, 
$\operatorname{Re}(z h'(z)/\mathcal{Q}(z)) > 0$ for each $\eta \geq \eta_{0}.$  Finally, by Lemma \ref{l7} we conclude that $p(z)\prec \cosh \sqrt{z}.$
\end{proof}
On taking $\eta=0$ in Theorem \ref{t14}, we deduce the following corollary.
\begin{corollary}
Suppose $p(z)$ is analytic in $\mathbb{D}$ with $p(z)\neq 0$ in $\mathbb{D}$ and $p(0)=1$ such that
\[\frac{zp'(z)}{p(z)}\prec \frac{\sqrt{z}}{2}\tanh \sqrt{z},\] then
$p(z)\prec \cosh \sqrt{z},$ and $\cosh \sqrt{z}$ is the best dominant.
\end{corollary}
On substituting $p(z)=zf'(z)/f(z)$ in Theorem \ref{t14}, we state the next corollary.                                                                                                   \begin{corollary}
Let $f\in\mathcal{A}$ and $\eta \geq \eta_{0},$ where $\eta_{0}$ is given by \eqref{e39}, such that 
\[1+\frac{zf''(z)}{f'(z)}-(1-\eta)\frac{zf'(z)}{f(z)}\prec \eta \cosh \sqrt{z}+\frac{\sqrt{z}}{2}\tanh \sqrt{z},\]
then $f\in\mathcal{S}^{*}_{\varrho}.$
\end{corollary}

\section{Some more results using admissibility conditions}
\noindent For $a\in\mathbb{C}$ and $n\in\mathbb{N},$ we now define a class of analytic functions $\mathcal{H}[a,n]$ as follows: \[\mathcal{H}[a,n]=\{f:\mathbb{D}\to \mathbb{C}:f \text{ is analytic; } f(z)=a+a_{n+1}z^{n+1}+\ldots\}.\] Clearly $\mathcal{H}_{1}:=\mathcal{H}[1,1].$ 
\begin{definition}
Let $\mathcal{Q}$ be the set of functions $q$ that are analytic and injective on $\overline{\mathbb{D}}\backslash E(q),$ where 
\[E(q)=\left\{\zeta\in\partial\mathbb{D}: \displaystyle{\lim _{z\rightarrow \zeta} q(z)=\infty}\right\},\] such that $q'(\zeta)\neq 0 $ for $\zeta \in \partial \mathbb{D}\backslash E(q).$ 
\end{definition} 
For $\Omega\subset\mathbb{C},$ $q\in\mathcal{Q}$ and $n\in\mathbb{N},$ let the class $\Psi_{n}(\Omega,q)$ consist of functions
$\psi:\mathbb{C}^{3}\times \mathbb{D}\to \mathbb{C}$ that meet the admissibility conditions: 
\begin{align*}
& \psi(r,s,t;z)\notin\Omega, \text{ whenever } (r,s,t;z)\in \mathbb{C}^{3}\times \mathbb{D}, \\& r=q(\xi), \text{ } s=m \xi q'(\xi), \text{ } \operatorname{Re}\left(1+\frac{t}{s}\right) \geq m \operatorname{Re}\left(1+\frac{\xi q''(\xi)}{q'(\xi)}\right),
\end{align*}
for $z\in\mathbb{D}, \xi\in\partial\mathbb{D}\backslash E(q)$ and $m\geq n.$ Denote $\Psi_{1}(\Omega,q)$ by $\Psi(\Omega,q).$

\begin{theorem}\cite[Theorem 2.3b]{Miller & Mocanu}\label{t15}
Let $\psi\in\Psi_{n}(\Omega,q)$ with $q(0)=a.$ If $p\in\mathcal{H}[a,n]$ satisfies 
\[\psi(p(z),zp'(z),z^{2}p''(z);z)\in\Omega,\] then $p(z)\prec q(z).$
\end{theorem}
If $\Omega \subsetneq \mathbb{C}$ is a simply connected domain, then there exists a conformal mapping $h(z)$ from $\mathbb{D}$ onto $\Omega = h(\mathbb{D}).$ Symbolically, let $\Psi_{n}(h(\mathbb{D}),q)$ represent $\Psi_{n}(\Omega,q).$ Further, if $\psi(p(z),zp'(z),z^{2}p''(z);z)$ is analytic in $\mathbb{D},$ then  $\psi(p(z),zp'(z),z^{2}p''(z);z)\in\Omega$ can be rewritten in terms of subordination: \[\psi(p(z),zp'(z),z^{2}p''(z);z)\prec h(z).\]
Consider $\Omega\subsetneq \mathbb{C}$ and  $q \in \mathcal{H}_{1}$ be given by $q(z):=\cosh \sqrt{z}.$ As $q(z)$ is univalent in $\overline{\mathbb{D}}\backslash E(q),$ where $E(q)=\emptyset,$  also $q(0)=1$ and $q(\mathbb{D})=\Omega_{\varrho},$ where $\Omega_{\varrho}$ is given by \eqref{e65}. Below we study the class of admissible functions $\psi_{n}(\Omega,q).$ \\ Note that for $|\xi|=1,$
\[q(\xi)\in q(\partial \mathbb{D}) = \partial \Omega_{\varrho} = \left\{\omega\in\mathbb{C}:|\log(\omega+\sqrt{\omega^{2}-1})|^{2}=1\right\}.\] Infact if $\xi=e^{i\theta},$ $-\pi<\theta\leq \pi,$ then
\[\xi q'(\xi) = \frac{\sqrt{\xi}}{2}\sinh \sqrt{\xi}, \text{ } q''(\xi) = \frac{1}{4 \xi}\left(\cosh \sqrt{\xi} - \frac{\sinh \sqrt{\xi}}{\sqrt{\xi}} \right)\]
and \[1+ \frac{ \xi q''(\xi)}{q'(\xi)} = \frac{1}{2}\left(1+ \sqrt{\xi} \coth \sqrt{\xi} \right).\] Further, it can be verified that the minimum value of $\operatorname{Re}(\sqrt{\xi} \coth \sqrt{\xi})$ is  attained at $\xi=-1.$ Thus the admissibility conditions with $q(z)=\cosh \sqrt{z},$ can be defined as follows:


\begin{definition}\label{def1}
Let $\Omega\subsetneq \mathbb{C}$ and $n\geq 1,$ then for $q(z)=\cosh \sqrt{z},$ the admissibility conditions are given as follows:
	\begin{align}\label{e56}
		\psi(r,s,t;z)\notin\Omega \text{ whenever }		& 
  \left\{
\begin{array}{lll}
			r=q(\xi)= \cosh \sqrt{\xi}, \\ \ \\ s= m \xi q'(\xi) =  \dfrac{m}{2}\sqrt{\xi}  \sinh \sqrt{\xi}, \\ \ \\
			\operatorname{Re}\left(1+\dfrac{t}{s}\right) \geq \dfrac{m}{2}(1+\cot 1),
		\end{array}
  \right.
	\end{align}
 where $z\in\mathbb{D},\xi\in\partial\mathbb{D}\backslash E(q)$ and $m \geq 1.$ We denote this class of admissible functions by $\Psi(q_{\varrho}).$
 \end{definition}
 In view of Theorem \ref{t15} and Definition \ref{def1}, we directly establish the next result:
 \begin{theorem}\label{t16}
 Let $p\in\mathcal{H}_{1}.$  
 \begin{enumerate}[(i)]
 \item 
 If $\psi\in \Psi(q_{\varrho}),$ then 
 $\psi(p(z),zp'(z),z^{2}p''(z);z)\in \Omega \Rightarrow p(z)\prec \cosh \sqrt{z}. $ \\
 \item If   $\psi\in \Psi(q_{\varrho}),$ with $\Omega=\Omega_{\varrho},$ then
$\psi(p(z),zp'(z),z^{2}p''(z);z)\prec \cosh \sqrt{z} \Rightarrow p(z)\prec \cosh \sqrt{z}.$
\end{enumerate}
 \end{theorem}
\noindent Recently insightful work is carried out in establishing several first and second order differential subordination implication results, using the concept of admissibility. For instance, many authors have studied the class of admissible functions associated with different analytic functions, such as: modified sigmoid function, lemniscate of Bernoulli, exponential function, petal shaped function, see \cite{Kumar & Goel(2020),Madaan(2019),Naz(2019),Neha & Sivaprasad(2023)}. Further, Kumar and  Goel \cite{Kumar & Goel(2020)}, modified the existing third order differential subordination results of Antonino and Miller \cite{Antonino & Miller (2011)}, in context of some special type of classes of starlike functions.
In the following results, we present a few applications to Theorem \ref{t16}.
\begin{theorem}\label{t17}
Let $p\in\mathcal{H}_{1},$ such that 
\[|zp'(z)-1|<\frac{\sin 1}{2}\approx 0.420\ldots,\]
then $p(z)\prec \cosh \sqrt{z}.$
\end{theorem}
\begin{proof}
 Suppose $\Omega=\{w:|w-1|<(\sin 1)/2\}.$ Let $\psi(p(z),zp'(z),z^{2}p''(z);z)$ be a function defined on $\mathbb{C}^{3}\times \mathbb{D},$ given by $\psi(r,s,t;z)=1+s.$
We need to show that for $(r,s,t)\in\mathbb{C}^{3}$ satisfies admissibility conditions given in \eqref{e56}. For $m\geq 1,$ consider
 \begin{align*}
 |\psi(r,s,t;z)-1|=|s|=\left|\frac{m}{2}\sqrt{\xi}\sinh \sqrt{\xi}\right|,
 \end{align*}
then for $\xi=e^{i\theta},$ where $-\pi< \theta \leq \pi,$ we have
  \begin{align*}
 |\psi(r,s,t;z)-1|=\frac{m}{2}\left|\sinh e^{i\theta/2} \right| \geq \frac{1}{2} \sin 1.
 \end{align*}
 This means $\psi(r,s,t;z)\notin\Omega$ for each $r,s,t$ satisfying \eqref{e56} and therefore, $\psi\in\Psi(\Omega,q).$ Finally, Theorem \ref{t16} leads to the required conclusion.
 \end{proof}

 \begin{theorem}\label{t18}
 Let $p\in\mathcal{H}_{1},$ such that 
\[\left|\frac{zp'(z)}{p(z)}\right|<\frac{\tanh 1}{2}\approx 0.380\ldots,\] then $p(z)\prec \cosh \sqrt{z}.$
 \end{theorem}
\begin{proof}
 Let $\Omega=\{w:|w|<(\tanh 1)/2\}$ 
and $\psi(p(z),zp'(z),z^{2}p''(z);z)$ be a function defined on $\mathbb{C}^{3}\times \mathbb{D},$ given by $\psi(r,s,t;z)=s/r.$
We need to show that for $(r,s,t)\in\mathbb{C}^{3}$ satisfying conditions \eqref{e56} leads to $\psi(r,s,t;z)\notin\Omega.$ Consider,
\begin{align*}
|\psi(r,s,t;z)|=\left|\frac{s}{r}\right|&=\left|\frac{m}{2}\sqrt{\xi}\tanh \sqrt{\xi}\right|
\geq \frac{m}{2}\tanh 1.
\end{align*}
Thus for $m\geq 1,$ we conclude that $\psi(r,s,t;z)\notin\Omega$ for each $r,s,t$ satisfying \eqref{e56}. Thus  $\psi\in\Psi(\Omega,q)$ and Theorem \ref{t16} gives that $p(z)\prec\cosh \sqrt{z}.$
 \end{proof}

 \begin{theorem}\label{t19}
 Let $p\in\mathcal{H}_{1},$ such that 
\[\left|\frac{zp'(z)}{p^{2}(z)}-1\right|<\frac{1}{2}\sech 1 \tanh 1 \approx 0.246\ldots,\] then $p(z)\prec \cosh \sqrt{z}.$
 \end{theorem}
\begin{proof}
 Suppose $\Omega=\{w:|w-1|<(  \sech 1 \tanh 1)/2\}.$ Let $\psi(p(z),zp'(z),z^{2}p''(z);z)$ be a function defined on $\mathbb{C}^{3}\times \mathbb{D},$ given by $\psi(r,s,t;z)=1+s/r^{2},$ then for $m\geq 1,$ we have
 \begin{align*}
 \left|\psi(r,s,t;z)-1\right|=\left|\frac{s}{r^{2}}\right|=\frac{m}{2}\left|\frac{\sinh \sqrt{\xi}}{\cosh^{2} \sqrt{\xi}}\right| 
 \geq \frac{\sinh 1}{2\cosh^{2}1}.
 \end{align*}
This gives that $\psi(r,s,t;z)\notin\Omega$ for each $r,s,t$ satisfying \eqref{e56}, therefore,  $\psi\in\Psi(\Omega,q).$ Thus Theorem \ref{t16} leads to the required conclusion.
\end{proof} 


On substituting $p(z)=zf'(z)/f(z)$ in Theorem \ref{t17} - Theorem \ref{t19}, we deduce the following:
\begin{corollary}
If $f(z)=z+a_{2}z^{2}+a_{3}z^{3}+\ldots$ satisfies any of the following inequalities:
\begin{enumerate}[(i)]
\item $\left| \dfrac{z^{2}f''(z)}{f(z)}+\dfrac{zf'(z)}{f(z)}-z^{2}\left(\dfrac{f'(z)}{f(z)}\right)^{2}-1\right|<\dfrac{\sin 1}{2}$ or
\item $\left|1+\dfrac{zf''(z)}{f'(z)}-\dfrac{zf'(z)}{f(z)}\right|<\dfrac{\tanh 1}{2}$  or
\item $\left|\left(1+\dfrac{zf''(z)}{f'(z)}-\dfrac{zf'(z)}{f(z)}\right)\left(\dfrac{zf'(z)}{f(z)}\right)^{-1}-1\right|<\dfrac{1}{2}\sech 1 \tanh 1,$ 
\end{enumerate}
then $f\in\mathcal{S}^{*}_{\varrho}.$
\end{corollary}

\end{document}